\documentclass[11pt, letterpaper]{article}

\usepackage{geometry}
\usepackage{color}
\usepackage{verbatim}
\usepackage[utf8x]{inputenc}
\usepackage{t1enc}
\usepackage[english]{babel}

\usepackage{amsmath,amssymb,amsthm}
\usepackage{mathrsfs,esint}
\usepackage{graphicx,wrapfig}
\usepackage[format=hang]{caption}
\usepackage{color}
\usepackage{soul}
\newcommand*{\Hcal}{{\mathcal H}_\infty}
\newcommand*{\Hcali}[1]{\mathcal{H}_{#1}}

\newcommand*{\R}{{\mathbb R}}

\newcommand*{\N}{{\mathbb N}}
\newcommand*{\Z}{{\mathbb Z}}

\newcommand*{\eps}{\varepsilon}

\newcommand*{\Om}{\Omega}

\providecommand*{\vint}[1]{\mathchoice
	{\mathop{\vrule width 5pt height 3 pt depth -2.5pt
			\kern -9pt \kern 1pt\intop}\nolimits_{\kern -5pt{#1}}}
	{\mathop{\vrule width 5pt height 3 pt depth -2.6pt
			\kern -6pt \intop}\nolimits_{\kern -3pt{#1}}}
	{\mathop{\vrule width 5pt height 3 pt depth -2.6pt
			\kern -6pt \intop}\nolimits_{\kern -3pt{#1}}}
	{\mathop{\vrule width 5pt height 3 pt depth -2.6pt
			\kern -6pt \intop}\nolimits_{\kern -3pt{#1}}}}
\newcommand*{\jint}{\fint}
\DeclareMathOperator{\Lip}{Lip}

\DeclareMathOperator{\dist}{dist}
\DeclareMathOperator{\diam}{diam}
\DeclareMathOperator{\rad}{rad}

\numberwithin{equation}{section}
\theoremstyle{plain}
\newtheorem{theorem}[equation]{Theorem}
\newtheorem{prop}[equation]{Proposition}

\newtheorem{lemma}[equation]{Lemma}

\theoremstyle{definition}

\newtheorem{definition}[equation]{Definition}

\newtheorem{example}[equation]{Example}
\newtheorem{constr}[equation]{Construction}
\begin{document}
	
\title{Traces of Newton-Sobolev functions on the visible boundary of domains in doubling metric measure spaces 
supporting a $p$-Poincar\'e inequality}
\author{Sylvester Eriksson-Bique, Ryan Gibara, Riikka Korte, Nageswari Shanmugalingam}
\maketitle
	
\begin{abstract} 
We consider the question of whether a domain with uniformly thick boundary at all locations and at all scales 
has a large portion of its boundary visible from the interior; here, ``visibility'' indicates the existence of John 
curves connecting the interior point to the points on the ``visible boundary''. In this paper, we provide an 
affirmative answer in the setting of a doubling metric measure space supporting a $p$-Poincar\'e inequality for 
$1<p<\infty$, thus extending the results of~\cite{Az, GK, KNN} to non-Ahlfors regular
spaces. We show that $t$-codimensional thickness of the boundary for $0<t<p$ implies $p$-codimensional 
thickness of the visible boundary. For such domains 
we prove that traces of Sobolev functions on the domain belong to the Besov class of the visible boundary.
\end{abstract}
	

\medskip
		
		
\section{Introduction}\label{Sec:1}

In the context of potential theory, geometric relationships between a domain (that is, an open
connected subset of a metric measure space) and its boundary play a key role. For instance, it was shown in~\cite{KL, L}
that when the ``visible part'' of the boundary of the domain is large in a codimensional sense, the domain
supports a Hardy-type inequality for Sobolev functions on the domain. Here, by the ``visible part'' of the boundary
we mean the points on the boundary that can be reached by a John curve from inside the domain, where John curves
are curves that satisfy a twisted cone condition. For smooth domains in Euclidean spaces, or more generally 
domains satisfying a non-tangential accessibility condition, all of the boundary is visible. However, for complicated 
domains whose boundaries can be fractal, cuspidal or multiply connected, it is much harder to show that the 
visible boundary is large. A natural way of measuring the size of the set is in terms of Frostman-like measures with a codimensional relationship. It was
shown in~\cite{Mal} that John domains whose boundary sustains a measure that is in a codimensional relationship
with the measure on the domain will support a trace theorem for Sobolev functions, with the trace on the boundary
of the domain belonging to a Besov class. It is therefore natural to ask whether there is an analogous trace theorem
for domains whose ``visible part'' of the boundary is large. 

From the works~\cite{Az, GK, KNN}, there is supporting
evidence that under some mild measure-theoretic conditions on 
its boundary, a domain
will have a large part of its boundary that is visible from the inside. 
In the setting of complete quasiconvex metric measure spaces with doubling measures, this fact was established 
recently in~\cite{EGKS}. 
However, the approach used in~\cite{EGKS}, while in the far more general context of doubling gauge
measures, did not establish the existence of a measure supported in the boundary of the domain.
The principal focus of the current paper is to construct a measure on the portion of the boundary of such a domain
that is ``visible'' from inside the domain, and under mild geometric conditions,  establish that there is a bounded
linear trace operator from the Newton-Sobolev class of functions on such a domain to the Besov class of functions on the 
``visible'' boundary; here the Besov class is with respect to the constructed boundary measure. 
To construct such a 
measure, we revisit the techniques found in~\cite{GK}, which addresses the visibility question in the context of domains
in metric spaces equipped with an Ahlfors regular measure supporting a Poincar\'e inequality. 
Albeit the slight lack of details in~\cite{GK}, the method there
is robust enough that we succeeded in modifying the technique of~\cite{GK} to the setting of metric measure spaces
supporting a Poincar\'e inequality, where the measure is only known to be doubling. In the first step we re-establish
the largeness of the ``visible part'' of the boundary, and in the process, we construct the measure on the boundary as a 
Frostman measure, and in the second step, we establish the existence of the bounded linear trace operator. 
These results are accomplished via the two
main theorems of this paper, Theorem~\ref{thm:main} and Theorem~\ref{thm:main2}. We discuss these in the next two subsections.

We also point out that~\cite{GK} assumed the measure to be Ahlfors regular, that is, there is a
positive constant $Q$ so that the measure of balls of radius $r>0$ is comparable to $r^Q$. This assumption was used
in a crucial way in~\cite{GK} to control the number of balls needed in a covering that is finitely chainable and contains a
well-placed collection of balls touching the boundary of the domain. In our setting of doubling measures, we do not have
this luxury. Thus our adaptation of the argument from~\cite{GK} diverges at this key point, and we circumvent such a lack
of control of the cover by developing an alternate form of control, and our point of divergence from~\cite{GK} 
is Lemma~\ref{lem:not-counting}. We also fix the few errors found in some of the proofs in~\cite{GK}, and for the convenience
of the reader, provide detailed proofs in the present paper.

\subsection{Size of the visible boundary}

For $c\ge 1$ and $z_0\in\Om$, let $\Om_{z_0}(c)$ be the largest subdomain of $\Om$
that is $c$-John with respect to the John center $z_0$.
For a fixed point $z_0\in\Omega$, the \emph{visible boundary} of $\Omega$ near 
$z_0$, $\partial\Omega_{z_0}(c)\cap\partial\Omega$, is the collection of all $w\in\partial\Omega$ for which 
there exists a rectifiable curve $\gamma$ with end points $z_0$ and $w$ such that 
 the length $\ell(\gamma_{z,w})$ satisfies
$\ell(\gamma_{z,w})\leq c\,d(z,\partial\Omega)$ for all $z\in\gamma$. Here, $\gamma_{z,w}$ 
denotes a segment of $\gamma$ joining $z$ to $w$, and $\gamma$ is known as a John curve. 
When $\Omega$ is a \emph{John domain} with John center $z_0$, that is, if every point in $\overline\Omega$ can be joined to 
$z_0\in\Omega$ by a John curve, then the visible boundary coincides with the topological boundary of the domain. 
When the domain is not John, the visible boundary will be strictly smaller. 

In~\cite{KNN}, it was shown that if $\Omega$ is a bounded simply connected domain in the complex plane, 
then  its visible boundary 
 (from near any fixed point in the domain) is large, that is, its $s$-dimensional Hausdorff content 
is comparable from below to a power of the distance from that fixed point in the domain to the boundary of the domain.
This holds for each $0<s<1$, where the ambient Euclidean plane is Ahlfors $2$-regular.
This result was extended to domains in higher dimensional Euclidean spaces in~\cite{Az}. It was shown 
there that
for domains $\Om\subset \R^n$ for which there exists $0<s<n-1$ such that
$\mathcal{H}^s_\infty(B(w,r)\setminus \Omega)\geq c_0\,r^s$ for all 
$w\in\partial\Omega$ and $0<r<\diam(\Omega)$, we can obtain an analogous control of 
$\mathcal{H}^\sigma(\partial\Omega_{z_0}(c)\cap\partial\Omega)$ in a manner similar to
that
found in~\cite{KNN} whenever $0<\sigma<s$.
In \cite{GK}, this study was extended to Ahlfors $Q$-regular metric measure spaces supporting a 
$p$-Poincar\'{e} inequality for some sufficiently small $p$. 

A natural measure 
of largeness in an Ahlfors regular space is given by Hausdorff content; however, when the underlying
measure on the metric space is merely doubling, a more natural measure of largeness of subsets of the metric space
is given by codimensional Hausdorff measure, see Definition~\ref{def:codim-Haus} below. In this paper,
the $t$-codimensional Hausdorff content of a set $A\subset X$ will be denoted by $\Hcal^{-t}(A)$,
and the corresponding restricted $t$-codimensional Hausdorff content (restriction to covers by balls of radius at most
$R>0$) will be denoted by $\mathcal{H}_{R}^{-t}(A)$. The $t$-codimensional Hausdorff measure
of $A$ is denoted by $\mathcal{H}^{-t}(A):=\lim_{R\to 0^+}\mathcal{H}_{R}^{-t}(A)$. 
When the measure on $X$ is Ahlfors $Q$-regular for some $Q>0$, then the $t$-codimensional Hausdorff measure
is equivalent to the $(Q-t)$-dimensional Hausdorff measure.
Even in the setting of Euclidean 
spaces, there are times when it is desirable to study potential theory
in domains equipped with an admissible weight as in~\cite{HKM}. Associated weighted measures are
rarely Ahlfors regular, but are always doubling.

From~\cite{BBS, CKKSS, GKS, GS2, Mal}, it is
clear that in the case of doubling metric measure spaces, traces of Sobolev spaces
are easier to come by if we have control over the codimensional Hausdorff measure of the boundary of
a domain in such a metric measure space. For this reason, it is desirable to obtain information regarding the
codimensional measure of the visible boundary of a domain.
		
The following is the first of the two main results of this paper. 
Here, for $x\in\Om$, we set $d_\Om(x):=\dist(x,\partial\Om)$.
		
\begin{theorem}\label{thm:main}
Let $(X,d,\mu)$ be a complete doubling geodesic metric measure space supporting a $p$-Poincar\'e inequality for some
finite $p> 1$. Let $0<t<p$ and $\Om$ be a domain in $X$, and let $c_0>0$ be such that
for all $w\in\partial\Om$ and for all $0<\rho<\diam(\Om)$, we have 
\begin{equation}\label{eq:Ass1}
	\Hcali{\rho}^{-t}(B(w,\rho)\cap\partial\Om)\ge c_0\, \frac{\mu(B(w,\rho))}{\rho^t}.
\end{equation}
Then for each $z_0\in\Om$ there is a compact subset $P_\infty\subset F:=B(z_0,3d_\Om(z_0))\cap\partial\Om_{z_0}(c)\cap\partial\Om$
and a probability measure $\nu_F$, 
called a Frostman measure, that is supported on $P_\infty$ and satisfies the following inequality:
there is a constant $c_2>0$ such that whenever $\zeta\in P_\infty$ and $0<r<3\,d_{\Om}(z_0)$, we have
\begin{equation}\label{eq:upper-bound-local-Frostman}
\nu_F(B(\zeta,r))\le c_2\, \frac{\mu(B(\zeta,r)\cap\Om)}{r^p}.
\end{equation}
Consequently,
there exist $c\ge 1$ and $c_1>0$ such that for each $z_0\in\Om$,
\begin{equation}\label{eq:Conc}
	\Hcali{d_\Om(z_0)}^{-p}(B(z_0,3d_\Om(z_0))\cap\partial\Om_{z_0}(c)\cap\partial\Om)
		\ge c_1\, \frac{\mu(B(z_0,d_\Om(z_0)))}{d_\Om(z_0)^p}.
\end{equation}
\end{theorem}

 This theorem will be proved in Section~\ref{Sec:4}.
 
 In light of the fact that for a fixed point $z_0\in\Om$
 we only obtain the control~\eqref{eq:Conc} and~\eqref{eq:upper-bound-local-Frostman}
 within the ball $B(z_0,3d_\Om(z_0))$,
 it is reasonable to ask whether we need the validity of~\eqref{eq:Ass1} for every $w\in\partial\Om$
 rather than just for $w\in B(z_0,M_0\, d_\Om(z_0))\cap\partial\Om$ provided we have fixed $z_0\in\Om$. Unfortunately
 such a local version of~\eqref{eq:Ass1} is not sufficient for our proof as the uniformity constant $c$ is only shown to exist,
 without explicit formula for $c$ solely in terms of $c_0$ and the doubling and Poincar\'e constants. Hence the curves
 we use in the proof can get close to $\partial\Om$ far away from $B(z_0,3d_\Om(z_0))\cap\partial\Om$, and so the constant
 $M_0$ needed cannot be pre-determined. For this reason we ask that~\eqref{eq:Ass1} holds for \emph{all} points
 $w\in\partial\Om$ even though, ultimately, we only need to know this inequality for points in $B(z_0, M_0\, d_\Om(z_0))$.

Note that the parameters $\rho$ and $d_\Omega(z_0)$ in the Hausdorff contents in \eqref{eq:Ass1} and \eqref{eq:Conc} are 
chosen to match the scale in question.  This way, the statement is scale-invariant. If $X$ is Ahlfors $Q$-regular, then 
$\mathcal{H}^{-p}_\rho(A)\approx \mathcal{H}^{-p}_\infty(A)$, and we can replace 
$\Hcali{\rho}$ and $\Hcali{d_\Omega(z_0)}$ with $\Hcali{\infty}$ in Theorem~\ref{thm:main}.

The requirement that $X$ be a geodesic space can be dispensed with. Indeed, the support of a Poincar\'e inequality guarantees
that $X$ is quasiconvex, and so a bi-Lipschitz change in the metric results in a geodesic space. By quasiconvexity it is meant that 
there is a constant $C_{\text{qua}}>1$ such that each pair of points $x,y\in X$ can be connected
by a curve of length at most $C_{\text{qua}}\, d(x,y)$. Hence, the
above theorem will continue to hold for spaces that are not geodesic, but in this case the constant $c_1$ and
the factor $3$ in~\eqref{eq:Conc} will change to other constants that now also 
depends on the quasiconvexity constant $C_{\text{qua}}$.

If the space $X$ satisfies a $p$-Poincar\'e inequality for some $p$ with
$1\leq p\leq t$, or if $0<t<p=1$, then
we can still apply Theorem~\ref{thm:main}, since the space also satisfies an 
$s$-Poincar\'e inequality for all $s>\max\{1,t\}$. 
In this case, 
we only obtain that
\begin{equation}\label{eq:Conc--p=1}
	\Hcali{d_\Om(z_0)}^{-s}(\partial\Om_{z_0}(c)\cap\partial\Om)
	\ge c_1\, \frac{\mu(B(z_0,d_\Om(z_0)))}{d_\Om(z_0)^{s}}.
\end{equation}

The requirement that $p>1$ in the statement of the theorem 
allows us to use~\cite{KZ} to obtain and apply a $q$-Poincar\'e inequality for some $1\le q<p$.
In this case, with $t<p$, we can choose such $q\in (t,p)$, and prove~\eqref{eq:Conc} with
$p$ replaced by $q+\eps$, where $\eps=p-q$, following the recipe first given in~\cite{GK}. 
By this reasoning, the main results in \cite{GK, KNN} can be obtained as corollaries. 
While we do allow for the possibility of $t\le 1$ in the above theorem, the conclusion~\eqref{eq:Conc} 
cannot give any information for the $\Hcali{d_\Om(z_0)}^{-q}$-content when $q<1$, 
see the counterexamples in~\cite[page~890]{Az}. 
If $t>1$ and $X$ satisfies a $p$-Poincar\'e inequality for some $p<t$, then it is not known whether the 
conclusion~\eqref{eq:Conc--p=1} holds with $s=t$.

\subsection{Trace result}

One of the principal players in the theory of the Dirichlet problem is the trace, of Sobolev functions on a domain,
at the boundary of the domain. The collection~\cite{Gris} and the paper~\cite{Strichartz} give a nice overview of
the role of traces in the study of Dirichlet problems, see also~\cite{RY} for the related Neumann boundary-value problem
for Waldenfels operators.

Traces of Sobolev functions were first studied by Gagliardo~\cite{Gag} with the
domain being the upper half-space in the Euclidean space $\R^n$, see also~\cite{Besov}. This was further extended to more general 
classes of Euclidean domains, assuming that the domains are either smooth or Lipschitz. Moreover,
the work of Jonsson and Wallin~\cite{JW} considered traces of Sobolev functions on $\R^n$ to subsets of
$\R^n$ that are the so-called $d$-sets with $d<n$. In the non-smooth setting of John domains in metric
measure spaces equipped with a doubling measure supporting a Poincar\'e-type inequality, an analogous
trace result was established by Mal\'y~\cite{Mal}. However, when the domain is not a John domain, the
aforementioned results do not apply to guarantee the existence of traces of Sobolev functions on the domain.
However, with the tools given by Theorem~\ref{thm:main}, we can construct traces on substantial parts of
the boundary near any given point in the domain. This is the second principal focus of the paper.

From Theorem~\ref{thm:main} we obtain the Frostman measure $\nu_F$, which is used to
gain control over
the set $P_\infty$  in order to obtain 
traces of Newton-Sobolev functions on $\Om$; this is the 
content of Theorem~\ref{thm:main2} below.
We will see in Section~\ref{Sec:5}
that by~\eqref{eq:upper-bound-local-Frostman}, whenever $u$ is a Newton-Sobolev function
in $\Om$, then $\nu_F$-almost every point in $P_\infty$ is a Lebesgue point of $u$, see
Lemma~\ref{lem:LebesguePts}. 

\begin{theorem}\label{thm:main2}
Suppose that $(X,d,\mu)$ and the domain $\Om\subset X$ satisfy the hypotheses of Theorem~\ref{thm:main}, that $\nu_F$ is as above, and in addition
we also suppose that there is some $\max\{1,t\}<p<\widehat{q}<q<\infty$ such that 
the measure $\mu\vert_\Om$ is doubling on $\Om$ and supports a $\widehat{q}$-Poincar\'e inequality.
Then there is a bounded linear trace operator 
\[
T:N^{1,q}(\Om)\to B_{q,q,2}^{1-p/q}(P_\infty,d,\nu_F)\cap L^q(P_\infty,\nu_F)
\]
such that $\|Tu\|_{B^{1-p/q}_{q,q,2}(P_\infty,d,\nu_F)}\lesssim \inf_{g_u}\int_\Om {g_u}^q\, d\mu$ and 
$\|Tu\|_{L^q(P_\infty,\nu_F)}\lesssim \|u\|_{N^{1,q}(\Om_{z_0}(c))}$,
where the infimum is over all upper 
gradients $g_u$ of $u\in N^{1,q}(\Om)$. The comparison constants in the above estimates are independent 
of $u$, and moreover the comparison constant associated
with the comparison between the Besov energy and the Sobolev energy is independent of $z_0$ as well.
\end{theorem}

See Definition~\ref{def:N1p+Btheta-p} for the definition of the relevant notions used in the above theorem, 
and~\eqref{eq:Thrace} for the explicit construction of the trace operator.
This theorem will be proved in Section~\ref{Sec:5}.

 As a consequence of Theorem~\ref{thm:main}, should a domain $\Om$ satisfy the hypotheses of
that theorem, then we can find a dense subset $E\subset\partial\Om$ and a countable decomposition $\{E_i\}_{i\in\N}$ of $E$
with each $E_i$ arising as the set $P_\infty$ given by Theorem~\ref{thm:main} for some point $z_0\in\Om$, and a Radon measure
$\nu_i$ on $E_i$ such that the trace $T$ of $N^{1,p}(\Om)$ to $E$ makes sense, though the trace function might not lie in a
global Besov class of functions on $E$. For Sobolev-type functions on the entire space $X$, similar traces to sets $S\subset X$
with a countable decomposition $\{S_i\}_{i\in\N}$ of $S$ into lower Ahlfors regular subsets of $S$ was explored in the
interesting papers~\cite{Tyu1, Tyu2}.

\subsection{Structure of the paper}

The paper is structured as follows. In Section~\ref{Sec:2},
we will discuss the background notation and preliminary results 
that are valuable tools in proving the two theorems mentioned above, and in
Section~\ref{Sec:3} we will construct the set $P_\infty$ and the Frostman measure $\nu_F$ supported on $P_\infty$.
Theorem~\ref{thm:main} is
then proved in Section~\ref{Sec:4}. A discussion of traces is provided, together with the proof of
Theorem~\ref{thm:main2}, in Section~\ref{Sec:5}. Section~\ref{Sec:6} contains some concluding remarks,
completing the paper.
		
\medskip
		
\noindent {\bf Acknowledgement:} 
Nageswari Shanmugalingam's work is partially supported by the NSF (U.S.A.) grant DMS~\#2054960. Sylvester Eriksson-Bique's 
work was partially supported by the Finnish Academy grant no. 356861. Riikka Korte's work was partially supported by the 
Finnish Academy grant no. 360184. Part of the research was conducted during S.~E.-B.'s
visit to the University of Cincinnati in December 2022, and he wishes to thank that institution for its kind hospitality.
The authors would also like to thank the anonymous referees for their input and corrections, which lead to the 
simplification of some of the arguments in this paper.

\section{Preliminaries}\label{Sec:2}

In this section, we provide definitions of notions used throughout this paper.
Here, $(X,d,\mu)$ is a complete metric measure space satisfying $0<\mu(B)<\infty$ for all balls $B$ of positive and finite radius.
We will also assume throughout the paper that $X$ is a geodesic space. We point out here that this
is a very mild requirement when $(X,d,\mu)$ is complete, 
doubling and supports a Poincar\'e inequality, for then $X$ is necessarily quasiconvex
and hence, there is a bi-Lipschitz equivalent metric on $X$ that turns $X$ into a geodesic space. A ball of radius $r>0$ will be denoted
by $B(r)$ if the center of the ball is not of particular importance to us.
		
\begin{definition}
We say that a domain $\Omega\subset X$ is \emph{c-John} for $c\geq 1$ with center $z_0$ if 
for each $y\in\Om$ with $y\neq z_0$ there exists a curve $\gamma$ with end points $z_0$ and $y$ such that
\[
\ell(\gamma_{z,y})\leq c\, d_\Omega(z)
\]
for all $z$ along $\gamma$. Such a curve is called a $c$-John curve with respect to $\Om$. 
Here, $\gamma_{z,y}$ denotes a subcurve of $\gamma$ with end points $z$ and $y$, and 
$d_\Omega(z):=d(z,X\setminus\Omega):=\inf\{d(z,\zeta)\, :\, \zeta\in X\setminus\Om\}$.
\end{definition}
		
It follows from the definition of a John domain that $\Omega$ is necessarily bounded. 
Moreover, points on the boundary of a John domain can also be reached from the center by a John curve. 
It is immediate from the definition of John domains and their centers that if $X$ is a geodesic
space and $\Om\subset X$ is a domain, $z_0\in\Om$, and $\Om_\tau$, $\tau\in\Lambda$, 
is a collection of $c$-John domains with
the same John center $z_0$, then $\bigcup_{\tau\in\Lambda}\Om_\tau$ is also $c$-John with the same John center. 
Given this, we know that
for $z_0\in\Om$ and $c\ge 1$ there is the largest subdomain, $\Om_{z_0}(c)$, of $\Om$, such that $\Om_{z_0}(c)$ is 
$c$-John with center $z_0$. 
Moreover, with $r=d_\Om(z_0)$, we have that $B(z_0,r)\subset\Om$ with $B(z_0,r)$ a $c$-John domain with John
center $z_0$ whenever $c\ge 1$. 
		
\begin{definition} 
Let $\Om$ be a domain in $X$.
For $c\ge 1$ and $z_0\in\Om$, we say that $\partial\Omega_{z_0}(c)\cap\partial\Omega$ is
the \emph{visible, or accessible, boundary} of $\Om$ near 
$z_0\in\Omega$. 		
\end{definition}
		
\begin{definition}\label{doubling}
We say that $\mu$ is a \emph{doubling measure} if there exists a constant $C_d\geq 1$ such that for all
$x\in X$ and $r>0$ we have
\[
\mu(B(x,2r))\le C_d\, \mu(B(x,r)).
\]
\end{definition}

\begin{definition} \label{def:codim-Haus}
For $t\geq 0$ and $R>0$, the {\emph{$t$--codimensional Hausdorff content}} of a set $A\subset X$ is defined as 
\[
\mathcal{H}^{-t}_{R}(A)
=\inf\left\{\sum_{i=1}^{\infty}\frac{\mu(B(x_i,r_i))}{r_i^t}\,:\,A\subset\bigcup_{i=1}^{\infty}B(x_i,r_i),\,r_i\leq R \right\}.
\]
The {\emph{$t$--codimensional Hausdorff measure} of $A$ is 
$\mathcal{H}^{-t}(A):=\lim_{R\to 0^+}\mathcal{H}^{-t}_R(A)$.}
\end{definition}

Recall that the Hausdorff dimension of a set $A$ in a metric space is the supremum of all $s\ge 0$ for which 
$\mathcal{H}^s_\infty(A)=\infty$, where $\mathcal{H}^s_\infty$ is the dimension $s$ Hausdorff content,
or equivalently, the supremum of all $s\ge 0$ for which $\mathcal{H}^s(A)>0$. Here, 
$\mathcal{H}^s(A):=\lim_{R\rightarrow 0^+}\mathcal{H}^s_R(A)$ is the usual $s$-dimensional 
Hausdorff measure. Unlike
the Hausdorff measures, Hausdorff content $\mathcal{H}^s_\infty(A):=\lim_{R\rightarrow \infty}\mathcal{H}^s_R(A)$ of 
a bounded set is never infinite; however, $\mathcal{H}^s(A)=0$ if and only if
$\mathcal{H}^s_\infty(A)=0$. When considering dimensional notions using codimensional measures, we 
have the following.
		
\begin{lemma} \label{lem:co-dim-change-exp}
Let $A\subset X$ be a bounded set and let $t>0$. Then for all $\tau\ge t$ we have that 
\[\rho^{\tau-t}\Hcali{\rho}^{-\tau}(A)\geq\Hcali{\rho}^{-t}(A).\] 
\end{lemma}
		
\begin{proof}
With $\{B_i\}_{i\in I\subset\N}$ a cover of $A$ by balls of radii $\rad(B_i)\le \rho$, we have that 
$\rho^{\tau-t}\, \rad(B_i)^t\ge \rad(B_i)^\tau$. Therefore,
\[
\Hcali{\rho}^{-t}(A)\le \sum_{i\in I}\frac{\mu(B_i)}{\rad(B_i)^t}\le \rho^{\tau-t}\sum_{i\in I}\frac{\mu(B_i)}{\rad(B_i)^\tau},
\]
and taking the infimum over all such covers yields that $\Hcali{\rho}^{-t}(A)\le \rho^{\tau-t}\, \Hcali{\rho}^{-\tau}(A)$.
\end{proof}
		
Unlike in the setting of Ahlfors regular
metric spaces equipped with the Ahlfors regular measure, in the doubling setting it is not necessarily true that when 
$\mathcal{H}^{-t}_\infty(A)=0$ we have $\mathcal{H}^{-t}(A)=0$, for large-scale behavior of $\mu(B(x,R))/R^t$, $x\in X$,
is not governed by the behavior of this ratio for small or intermediate values of $R$. It is therefore more instructive to
consider the scaled Hausdorff codimensional contents $\Hcali{R}^{-t}$, $R>0$. The next lemma gives a partial control
over how the change in the scale $R$ affects the codimensional Hausdorff content. 
		
\begin{lemma}\label{lem:Hausdorff-scale}
Let $\alpha>\rho$ and $t>0$. Then, whenever $K\subset X$ and $\{B_i\}_{i\in I\subset\N}$
is a cover of $K$ by balls $B_i$ of radii $r_i\le \alpha$, we have
\[
\Hcali{\rho}^{-t}(K)\le C\left(\frac{\alpha}{\rho}\right)^t\, \sum_{i\in I}\frac{\mu(B_i)}{r_i^t}.
\]
Here $C$ depends only on the doubling constant $C_d$.
\end{lemma}

\begin{proof}
For a cover $\{B_i\}_{i\in I\subset\N}$ of
$K$ comprised of balls $B_i$ of radius $r_i\le \alpha$, we can decompose this cover into two subfamilies: 
the first subfamily
$\{B_i\}_{i\in I_1\subset I}$ of all the balls $B_i$ with radius $r_i\le \rho$, 
and the other subfamily
$\{B_i\}_{i\in I_2\subset I}$  with $I_2=I\setminus I_1$ consisting
of balls  of radius $r_i\ge\rho$. For each ball $B_i$ with $i\in I_2$, we can find
a collection of balls $\{B_{i,j}\}_{j\in I(i)\subset \N}$ contained in $2B_i$ with radius $\rho$ 
so that they have bounded overlap (with the overlap 
bounded by a constant that depends solely on the doubling constant of $\mu$) and cover $B_i$. For each $i\in I_2$, then, we have 
\[
\sum_{j\in I(i)}\frac{\mu(B_{i,j})}{\rho^t}\le C\, \frac{\mu(2B_i)}{\rho^t}
\le C\, \left(\frac{\alpha}{\rho}\right)^t\, \frac{\mu(B_i)}{r_i^t}.
\]
The collection $\{B_i\}_{i\in I_1\subset I}\cup\{B_{i,j}\}_{i\in I_2,j\in I(i)}$ is then a cover of $K$ comprised of balls with radii at most $\rho$, and so 
\[
\Hcali{\rho}^{-t}(K)\leq \sum_{i\in I_1}\frac{\mu(B_i)}{r_i^t}+C\, \left(\frac{\alpha}{\rho}\right)^t\sum_{i\in I_2}\frac{\mu(B_i)}{r_i^t}\leq C\, \left(\frac{\alpha}{\rho}\right)^t \sum _{i\in I} \frac{\mu(B_i)}{r_i^t}. \qedhere
\]
\end{proof}

For a locally Lipschitz function $u$ on $X$, for $x\in X$ we set
\[
\Lip(u)(x):=\limsup_{r\rightarrow 0^+}\sup_{y\in B(x,r)}\frac{|u(x)-u(y)|}{d(x,y)}. 
\]
		
\begin{definition}\label{Poin}
We say that $X$ supports a \emph{$p$-Poincar\'e inequality}, $1\leq p<\infty$, if there exist 
constants $C>0$ and $\kappa\ge 1$ such that for all $x\in X$ and $r>0$ we have
\[
	\fint_{B(x,r)}\!|u-u_{B(x,r)}|\, d\mu\le C\, r\, \left(\fint_{B(x,\kappa r)}\!\Lip(u)^p\, d\mu\right)^{1/p}
\]
for all locally Lipschitz functions $u$, where 
\[
u_{B(x,r)}:=\fint_{B(x,r)}\!u\, d\mu:=\frac{1}{\mu(B(x,r))}\int_{B(x,r)}\!u\, d\mu.
\]
\end{definition}
		
On a complete metric measure space, the $p$-Poincar\'e inequality for $1<p<\infty$ exhibits self-improvement. 
It follows from a theorem of Keith and Zhong, see~\cite{E-B,KZ}, that the space also satisfies a 
$q$-Poincar\'e inequality for some $1\leq q<p$. 

\begin{definition}\label{def:N1p+Btheta-p}
Given a metric measure space $(Y,d_Y,\mu_Y)$ with $\mu_Y$ Borel, we say that a function $u:Y\to[-\infty,\infty]$
is in the Newton-Sobolev class $N^{1,q}(Y)$ if $u$ is measurable, $\int_Y|u|^q\, d\mu<\infty$, 
and there is an upper gradient $g\in L^q(Y)$; that is,
$g$ is a non-negative Borel function such that 
for every non-constant compact rectifiable curve $\gamma$ in $Y$, we have 
\[
|u(y)-u(x)|\le \int_\gamma g\, ds,
\]
where $x$ and $y$ denote the end points of $\gamma$ in $Y$. The norm on $N^{1,q}(Y)$ is given by
\[
\Vert u\Vert_{N^{1,q}(Y)}:=\left(\int_Y|u|^q\, d\mu_Y +\inf_g\int_Yg^q\, d\mu_Y\right)^{1/q},
\]
with the infimum over all upper gradients  $g$ of $u$.
	
For $S\ge 1$ we say that $u\in B^\theta_{q,q,S}(Y,d_Y,\mu_Y)$ if the nonlocal energy (semi)norm, given by
\[
\|u\|_{B^\theta_{q,q,S}(Y,d_Y,\mu_Y)}^q:=
\int_Y\int_Y\frac{|u(y)-u(z)|^q}{d_Y(y,z)^{\theta q}\, \mu_Y(B(z,S\, d_Y(y,z)))}\, d\mu_Y(y)\, d\mu_Y(z)
\]
is finite. 
\end{definition}

Note that when the measure $\mu_Y$ is doubling, 
the spaces $B^\theta_{q,q,S}(Y,d_Y,\mu_Y)$ are independent of the parameter $S$,
with the energy seminorms comparable with comparison constant depending solely on the doubling constant of $\mu_Y$ and
the parameter $S$. In this case,
we can set $S=1$ in the above definition to obtain the classical Besov class 
$B^\theta_{q,q}(Y,d_Y,\mu_Y):=B^\theta_{q,q,1}(Y,d_Y,\mu_Y)$, with equivalent norms. However, in
our setting, the measure $\mu_Y$ is the \emph{upper} $p$-codimensional Hausdorff measure 
$\nu_F$ on a compact subset $P_\infty$ of the
boundary of the domain, and we do not know that this measure is doubling. 
For nondoubling measures 
$\mu_Y$,  changing the parameter $S$
could result in non-controllable changes in the Besov class and in the associated energy.
	
Besov spaces arise naturally in the study of potential theory, especially in understanding the 
class of boundary data for which solutions to the $p$-energy minimizing problem (equivalent to
the $p$-Laplacian problem in Euclidean, Riemannian, and Carnot-Carath\'eodory settings) 
exist with the corresponding boundary data.
Of such boundary data, those that give rise to solutions that have globally finite 
$p$-energy correspond precisely to traces of Newton-Sobolev functions on the domain.
We refer interested readers to~\cite{BBS, CKKSS, GKS, GS2, GoKS} for more information on 
Besov spaces as the trace class of Sobolev spaces.

\vskip .3cm
		
\noindent{\bf Standing assumptions:} 
\emph{Throughout this paper, apart from Section~\ref{Sec:5}, we assume that $(X,d,\mu)$ is a complete 
metric measure space, that $\mu$ is a doubling measure, 
and that it supports a $p$-Poincar\'e inequality for some $1<p<\infty$. In this setting, $(X,d)$ is proper and $d$ is bi-Lipschitz 
equivalent to a geodesic metric, see~\cite[Lemma~4.1.14, Corollary~8.3.16]{HKSTbook}. 
We take $d$ to be this geodesic metric, which implies in 
turn that one may take the Poincar\'e scaling factor $\kappa$ to equal $1$ (see~\cite[Remark~9.1.19]{HKSTbook}); 
this is crucial for us only in the proofs of Lemmas~\ref{lem:Loewner} and~\ref{lem:not-counting} 
below. Moreover, the parameter
$q$ will be used to denote the self-improved Poincar\'e inequality 
exponent, i.e.~$q$-Poincar\'e inequality, guaranteed by~\cite{KZ}.}

\section{Construction of $P_\infty$ and the Frostman measure $\nu_F$}\label{Sec:3}
		
Recall that $(X,d,\mu)$ is complete, doubling, and supports a $p$-Poincar\'e inequality for some $p>1$, and $0<t<p$. 
Thus by the self-improvement property due to~\cite{KZ}, we can find some $q\ge 1$ with $t<q<p$ such that $X$ also 
supports a $q$-Poincar\'e inequality.
		
Replacement for~\cite[Lemma~4.3]{GK} is the following, and its proof is an adaptation of~\cite[Theorem~5.9]{HK}.
We provide the complete adapted proof here, especially since we do not require the sets $E$ and
$F$ to be compact. 
		
\begin{lemma}\label{lem:Loewner}
In the setting of Theorem~\ref{thm:main}, 
fix $1<q<p$ such that $X$ supports a $q$-Poincar\'e inequality.
Let $0<t<q$. 
Then there is a constant $C\ge 1$ such that if $r>0$, $E$ and $F$ are 
subsets of a ball $B(x,r)\subset X$, and $\lambda>0$ satisfies
\[
\min\{\Hcali{r}^{-t}(E), \Hcali{r}^{-t}(F)\}\ge \lambda \frac{\mu(B(x,r))}{r^t}, 
\]
then for any Lipschitz function $u$ on $B(x,r)$ with $u=1$ on $E$ and $u=0$ on $F$, we have that 
\[
\int_{B(x,r)}(\Lip u)^q\, d\mu\ge \frac{\lambda}{C} \frac{\mu(B(x,r))}{r^q}.
\]
\end{lemma}
		
\begin{proof} 
By replacing $u$ with $\min\{\max\{u,0\},1\}$ if necessary, we may assume that $0\le u\le 1$ on $B(x,r)$.
In the case that $u_{B(x,r)}\le \tfrac12$, 
for each $z\in E$ we have $|u(z)-u_{B(x,r)}|\ge \tfrac12$; in the case that 
$u_{B(x,r)}\ge \tfrac12$, for each $w\in F$ we have $|u(w)-u_{B(x,r)}|\ge \tfrac12$. 
We will address the first case, with the second case following {\it mutatis mutandis}.
			
We set $\beta:=1-\tfrac{t}{q}$, and note that as $t<q$, necessarily $\beta>0$.
For each $z\in E$ let $[x,z]$ be a geodesic segment with end points $x$ and $z$. 
If $d(x,z)\ge r/4$ we let $B_0:=B(x,r)$, $B_1=B(x,d(x,z)/2)$,
and let $x_2$ be the point on the segment $[x,z]$ that is a distance $d(x,z)/2$ from $x$. We set $B_2=B(x_2,d(x,z)/2^2)$.
Inductively, once we have chosen $B_i$, $i=0,1,\cdots, k$, we set $x_{k+1}$ to be the point on the segment $[x,z]$
that is at distance $2^{-j}d(x,z)$ from $z$. This implies that $d(x_{k+1},x)=\sum_{j=1}^k2^{-j}d(x,z)$. Now we set 
$B_{k+1}=B(x_{k+1},2^{-(k+1)}d(x,z))$. 
If $d(x,z)<r/4$ then we set $B_0:=B(x,r)$, $B_1:=B(z,r/2)$, and for positive integers $k$ we set $B_k=B(z,2^{-k}r)$.
			
Then note that for each 
positive integer $k$, we have $2B_k\subset B(x,r)$, 
$B_{k+1}\subset 2B_{k}$, and $z\in 3B_k$. Therefore, by the doubling property of $\mu$, we get
\begin{align*}
		\frac{1}{2}<|u(z)-u_{B(x,r)}|\le \sum_{k=0}^\infty |u_{B_k}-u_{B_{k+1}}|
				&\lesssim \sum_{k=0}^\infty \jint_{2B_k}|u-u_{2B_k}|\, d\mu\\
				&\lesssim \sum_{k=0}^\infty 2^{-k}r\left(\jint_{2B_k}(\Lip u)^q\, d\mu\right)^{1/q}.
\end{align*}
Since $\sum_{k=0}^\infty 2^{-k\beta}$ is a convergent series, it follows that there is a non-negative integer $k_z$ such that
\[
2^{-k_z\beta}\lesssim 2^{-k_z}r\left(\jint_{2B_{k_z}}(\Lip u)^q\, d\mu\right)^{1/q}.
\]
From our choice of $\beta$, we have that $t=(1-\beta)q$.  From the above inequality we obtain
\[
\mu(B_{k_z})\lesssim r^q 2^{-k_z(1-\beta)q}\int_{2B_{k_z}}(\Lip u)^q\, d\mu
			\lesssim r^{\beta q}\, \rad(B_{k_z})^t \, \int_{2B_{k_z}}(\Lip u)^q\, d\mu.
\]

The family $\{3B_{k_z}\, :\, z\in E\}$ forms a cover of $E$, and so by the
$5B$-covering lemma we can obtain a countable
pairwise disjoint subcollection $\{3B_k\}_{k\in I\subset\N}$ such that $\{15B_k\}_{k\in I\subset\N}$ also covers $E$.
It follows from Lemma~\ref{lem:Hausdorff-scale} that
\begin{align*}
\Hcali{r}^{-t}(E)\lesssim \sum_{k\in I}\frac{\mu(15B_k)}{(15\rad(B_k))^t}
	&\lesssim \sum_{k\in I}\frac{\mu(B_k)}{(\rad(B_k))^t}\\
	&\lesssim\sum_{k\in I} r^{\beta q}\int_{2B_k}(\Lip u)^q\, d\mu \le r^{\beta q}\int_{B(x,r)}(\Lip u)^q\, d\mu.
\end{align*}
By hypothesis, we know that $\mathcal{H}^{-t}_r(E)\ge \lambda r^{-t}\, \mu(B(x,r))$. As $-t-\beta q=-q$, the desired conclusion follows.
\end{proof}

\begin{definition}\label{def:well-placed}
Let $B$ be a ball in $X$ with center in 
$\overline{\Om}$. We say that a family $\{B_i\}_{i=1}^N$ of balls is well-placed
along $B\cap\partial\Om$ if for each $i=1,\cdots, N$ we have
that $B_i\subset\Om$, $(B\cap\partial\Om)\cap\partial B_i$ is non-empty, and
$4B_i\cap 4B_j$ is empty whenever $i\ne j$ with $i, j\in\{1,\ldots, N\}$. 
\end{definition}

A chain of balls is a collection $B_1,\dots, B_N$ such that $\rad(B_i)=\rad(B_1)$ and  $B_i\subset \Omega$ for all $i=1,\dots, N$ and the center of $B_{i+1}$ lies in $\overline{B_{i}}$ for $i=1,\dots, N-1$.

We say that a ball $B\subset \Omega $ is chainable to another ball $B' \subset \Omega$ if $\rad(B)=\rad(B')$ and there exists an $N\in \N$ and a chain of balls $B_1=B,\dots, B_N=B'$. We further say that $B$ is chainable to $B'$ within a set $Y$ if $B_i\subset Y$ for each $i=2,\dots, N-1$.

\begin{lemma}\label{lem:not-counting}
	Let 
	$w\in\partial\Om$, $0<r_0<\diam(\Om)$, and $0<\eta_0<\tfrac{1}{168}$. Suppose that for $\rho=(1-21\eta_0)r_0$ we have 
	\[
	\Hcali{r_0}^{-t}(B(w,\rho)\cap\partial\Om)\ge c_0\, \frac{\mu(B(w,\rho))}{\rho^t},
	\]
	and that $z\in\Om\cap\partial B(w,r_0)$ is such that $B(z,r_0/2)\subset\Om$. Let $\{B_i(\eta_0 r_0)\}_{i=1}^N$ be a maximal  
	collection of balls with centers in $B(w,r_0)\cap\Om$ s.t. each $B_i$ is chainable to $B(z,\eta_0 r_0)$ within 
	$B(w,2r_0)\cap\Om$ and $\{B_i(\eta_0 r_0)\}_{i=1}^N$ is well-placed along $B(w,2r_0) \cap \partial \Omega$. 
	Then
	\[
	\eta_0^q\, \mu(B(z,r_0))\le K_0\sum_{i=1}^N\mu(B_i(\eta_0 r_0)),
	\]
	with the constant $K_0$ dependent solely on the doubling and Poincar\'e constants of $X$.
\end{lemma}

\begin{proof}
	We now set $D$ to be the collection of all points $x\in\Om$ with $d_\Om(x)> \eta_0r_0$ 
	for which $B(x,\eta_0 r_0)$ is chainable to $B(z,\eta_0r_0)$. 
	We consider the functions $g:X\to\R$ and $f:B(w,r_0)\to\R$  given by 
	\[
	g(x):=\max_{1\le i\le N}\left(1-\frac{\dist(x,30B_i)}{\eta_0 r_0}\right)_+
	\]
	and
	\[
	f(x)=\begin{cases} g(x) &\text{ if }x\in B(w,r_0)\cap D, \\ 
		1 &\text{ if }x\in B(w,r_0)\setminus D. \end{cases} 
	\]
	It is direct to see that $f$ is $\tfrac{1}{\eta_0 r_0}$-Lipschitz on $B(w,r_0)\cap D$, and that $f$ is constant 
	(and hence Lipschitz
	continuous) on $B(w,r_0)\setminus\overline{D}$. Thus, to show that $f$ is locally Lipschitz continuous 
	on $B(w,(1-21\eta_0)r_0)$, it suffices 
	to show that $f$ is locally Lipschitz continuous at each point $x$ in $B(w,(1-21\eta_0)r_0)\cap \partial D$.  It is direct to see 
	that $d_\Om(x)\geq \eta_0\, r_0$, since this estimate holds for all points in $D$. Since $x\in \partial D$ we have some $\tilde{x}\in D$ 
	such that $d(x,\tilde{x})< \eta_0 r_0$. Since $\tilde{x}\in D$, we have that $B(\tilde{x},\eta_0 r_0)$ is chainable to $B(z,\eta_0r_0)$ 
	within $B(w,2r_0)\cap\Om$. By adding $B(x,\eta_0 r_0)$ to the chain, we see that $B(x,\eta_0r_0)$ is also chainable to 
	$B(z,\eta_0 r_0)$ within $B(w,2r_0)\cap\Om$. 
	
	There are two cases to consider. First, assume $d_\Om(x)>\eta_0 r_0$. By again adjoining a ball, we see that for each $y\in B(x,d_\Om(x)-\eta_0\, r_0)$ the ball $B(y,\eta_0 r_0)$ is chainable to $B(z,\eta_0r_0)$. Thus, $B(x,d_\Om(x)-\eta_0\, r_0)\subset D$, which is a contradiction to $x\in \partial D$. Therefore $d_\Om(x)=\eta_0\, r_0$.
	
	Thus, we focus on points $x\in B(w,(1-21\eta_0)r_0)$ for which $d_\Om(x)=\eta_0\, r_0$. From the above discussions,
	we know that the ball $B(x,\eta_0 r_0)$ is chainable 
	to $B(z,\eta_0r_0)$ within $B(w,2r_0)\cap\Om$. By maximality of the well-placed collection, we must have some $i=1,\dots, N$ such that
	 $4B_i \cap 4B(x,\eta_0 r_0) \neq \emptyset$. Hence 
$B(x,\eta_0 r_0)\subset 12B_i \subset 30 B_i$. Therefore, $f=1$ in the ball $B(x,\eta_0 r_0)$, and the local Lipschitz property follows. In conclusion, $f$ is locally $\tfrac{1}{\eta_0 r_0}$-Lipschitz continuous on $B(w,(1-21\eta_0)r_0)$. 
	
Now let $E=B(z,r_0/4)\cap B(w,(1-21\eta_0)r_0)\cap D$ and $F=B(w,(1-21\eta_0)r_0)\cap\partial\Om$. 
		Recall that $X$ is a geodesic space.
		Let $\gamma$ be a geodesic connecting $z$ to $w$, and let $z''$ be a point along this geodesic so that
		$d(z'',z)=3r_0/16$, whence we also have that $d(z'',w)=13r_0/16$. We claim that $B(z'',r_0/16)\subset E$. 
		By the triangle inequality we know
		that $B(z'',r_0/16)\subset B(z,r_0/4)\cap B(w,(1-21\eta_0)r_0)$. 
		Thus, to show that $B(z'',r_0/16)\subset E$ it remains to show that $B(z'',r_0/16)\subset D$. We do this by constructing 
		a chain of balls connecting balls centered in this set to $B(z,\eta_0 r_0)$. Take any point 
		$x\in B(z'',r_0/16)\setminus D$, and let $\beta$ be the concatenation of a geodesic
		from $x$ to $z''$ and the subcurve of $\gamma$ from $z''$ to $z$. Choose $x_1=z, x_k=x$ and points $x_2,\dots, x_{k-1}$ along 
		$\beta$ with $d(x_2,z)=\eta_0 r_0$ and $d(x_i,x_{i+1})\leq \eta_0 r_0$ for $i=1,\dots, k-1$. Recall that $d_\Om(z)\ge r_0/2$ and
		$\eta_0<\tfrac{1}{168}$. It follows that $\{B(x_i,\eta_0r_0)\}_{i=1}^k$ 
		gives the chain connecting $B(x,\eta_0\, r_0)$ to $B(z,\eta_0\, r_0)$. Thus, $x\in D$ by definition.
			
	The sets $E$ and $F$ are disjoint subsets of the 
	ball $B(w,(1-21\eta_0)r_0)$. Moreover, by the hypothesis of the lemma, we know that 
	\begin{equation}\label{eq:F-bd}
		\Hcali{r_0}^{-t}(F)\ge c_0\, \frac{\mu(B(w,(1-21\eta_0)r_0))}{r_0^t}\gtrsim\frac{\mu(B(w,r_0))}{r_0^t}.
	\end{equation}
	Recall that $0<\eta_0<\tfrac{1}{168}$.
	Setting $B'':=B(z'',r_0/16)$ where $z''$ was the point on the geodesic from $z$ to $w$ as described above,
		that argument above demonstrated that $B''\subset E$. Therefore,
	we also have that
	\begin{equation*}
		\Hcali{r_0}^{-t}(E)\ge \Hcali{r_0}^{-t}(B''). 
	\end{equation*}
	To compute $\Hcali{r_0}^{-t}(B'')$ we argue as follows. Let $\{B^*_i\}_{i\in I\subset\N}$ be a cover of $B''$ by balls $B^*_i$ 
	of radii $r_i\le r_0$. Since for each $i\in I$ we have that $r_i\le  r_0$, then by the doubling property of $\mu$ we have
	\[
	\sum_{i\in I}\frac{\mu(B^*_i)}{r_i^t}\ge \sum_{i\in I}\frac{\mu(B^*_i)}{r_0^t}
	\ge \frac{\mu(B'')}{r_0^t}\gtrsim\frac{\mu(B(w,r_0))}{r_0^t}.
	\]
 It follows that 
		\begin{equation}\label{eq:E-bd}
			\Hcali{r_0}^{-t}(E)\ge \Hcali{r_0}^{-t}(B'')\gtrsim \frac{\mu(B(w,r_0))}{r_0^t}.
	\end{equation}
	
	We next show that $f=1$ on $F$ and that $f=0$ on $E$, setting us up to use Lemma~\ref{lem:Loewner}. Indeed, 
	by construction, we have $E\subset D\cap B(w,r_0)$, and so $f=g$ on $E$. 
		If $x\in E$, then $\dist(x,\partial\Om)\ge \tfrac{r_0}{4}$, and so for each $i=1,\cdots, N$ we have that 
	$\dist(x,31B_i)\ge \tfrac{r_0}{4}-32\eta_0 r_0>0$ since $\eta_0<\tfrac{1}{128}$. Therefore necessarily $g=0$ on $E$,
	that is, $f=0$ on $E$.  Since $F\cap D =\emptyset$, we have $f|_F=1$.
	
	We then get to
	\[
	\int_{B(w,(1-21\eta_0)r_0)}(\Lip f)^q\, d\mu\le \frac{1}{(\eta_0 r_0)^q}\mu\left(\bigcup_{i=1}^N 31B_i(\eta_0 r_0)\right)
	\lesssim \frac{1}{(\eta_0 r_0)^q}\sum_{i=1}^N\mu(B_i(\eta_0 r_0)).
	\]
	Now an application of Lemma~\ref{lem:Loewner} (with $\lambda$ determined by $c_0$ and the doubling constant
	related to $\mu$) above, together with the fact that $r_0>(1-21\eta_0)r_0\ge r_0/2$ gives
	\[
	\frac{1}{C} c_0 \frac{\mu(B(w,r_0))}{r_0^q}\le \frac{1}{(\eta_0 r_0)^q}\sum_{i=1}^N\mu(B_i(\eta_0 r_0)).
	\]
	Setting $K_0=C/c_0$ yields the desired conclusion.
\end{proof}

Next we wish to construct a Frostman measure, adapted to the use of codimensional Hausdorff content. 
The idea comes from~\cite{VK}, but unlike in~\cite{VK}, we are not interested in creating a doubling measure,
and so our construction is simpler and closer to that of~\cite{GK}.  Henceforth we consider
$0<\eta<\eta_1<\tfrac{1}{168}$, with $\eta_1$ sufficiently small as determined in the proof of Proposition~\ref{prop:Frostman}.
	
\begin{constr}\label{def:P-infty}
As in~\cite{GK}, fix $z_0\in\Om$ and let $r=d_\Om(z_0)$. We then choose a point $w_0\in\partial B(z_0,r)\cap\partial\Om$ and
set $P_0=\{w_0\}=\mathcal{W}_0$. We begin the construction 
by choosing a point $x_1$ in the geodesic  connecting $z_0$ to $w_0$ such that $d(x_1,w_0)=\eta r$. The ball
$B(x_1,\eta r)$ is chainable to $B(z_0, \eta r)$ by balls that are chosen with centers on the geodesic connecting $x_1$ to $z_0$.
Since $d_\Om(z_0)=r$, it follows that $d_\Om(x_1)=\eta r$.
Next, as in the statement of Lemma~\ref{lem:not-counting} with $z_0$ playing the role of $z$, $w_0$ playing the role 
of $w$, $r$ playing the
role of $r_0$, and $\eta$ playing
the role of $\eta_0$, let $\mathcal{W}_1=:\mathcal{W}_1(w_0)$ 
be a maximal collection of balls of radius $\eta r$ with the properties:
\begin{enumerate}
\item  $B(x_1,\eta r)\in \mathcal{W}_1(w_0)$,
\item it is well-placed along 
$B(w_0,2r)\cap\partial\Om$, 
\item each ball in this collection is finitely
chainable to $B(z_0,\eta r)$ by balls within $\Om\cap B(w_0,2r)$. 
\end{enumerate}
Let 
$P_1(w_0)$  consist of a choice of one point in
$B(w_0,2r)\cap\partial B\cap\partial\Om$ for each $B\in \mathcal{W}_1$. For $x_1$, we choose $w_0$. 
From the conditions related to $\mathcal{W}_1(w_0)$, it follows that $P_1(w_0)$ has as many
distinct points as the number of distinct balls in the collection $\mathcal{W}_1(w_0)$. 
We
also set $P_1=\bigcup_{w\in P_0}P_1(w)$, which, in this case, equals $P_1(w_0)$.
By construction, $P_0\subset P_1$. 
	
We inductively construct $\mathcal{W}_{k+1}$,  and 
$P_{k+1}$, given the successful construction of $\mathcal{W}_{k}$ and
$P_k$ as follows: for each $w\in P_k$, there is a ball $B_w\in \mathcal{W}_k$ such that 
$w\in\partial B_w\cap\partial\Om$. Let $z_w$ be the center of $B_w$. We 
choose $x_w$ in the geodesic connecting $z_w$ and $w$ such that $d(x_w, w)=\eta^{k+1} r$. 
Note that $d_\Om(z_w)=\eta^k r$, and so $d_\Om(x_w)=\eta^{k+1}r$.
As in Lemma~\ref{lem:not-counting}, with $w$ playing the role of $w$, $z_w$ playing the
role of $z$, $\eta^k r$ playing the role of $r_0$, and $\eta$ playing the role of $\eta_0$ there, we set
$\mathcal{W}_{k+1}(w)$ to be a maximal collection of balls with
\begin{enumerate}
\item $B(x_w,\eta^{k+1}r)\in \mathcal{W}_{k+1}(w)$,
\item it is well-placed along 
$B(w,2\eta^{k}r)\cap\partial\Om$, 
\item each ball in this collection is finitely
chainable to $B(z_w,\eta^{k+1} r)$ by balls within $\Om\cap B(w_k,2\eta^{k}r)$. 
\end{enumerate}
Let $P_{k+1}(w)$ consist of a choice of a point in $B(w,2\eta^k r)\cap\partial B\cap\partial\Om$ for each 
$B\in \mathcal{W}_{k+1}(w)$
with the additional requirement that $w\in P_{k+1}(w)$, which ensures that $P_k\subset P_{k+1}$. 
We now set $\mathcal{W}_{k+1}:=\bigcup_{w\in P_k}\mathcal{W}_{k+1}(w)$, and $P_{k+1}=\bigcup_{w\in P_k}P_{k+1}(w)$.
		
We set $P_\infty$ to be the closure of $\bigcup_kP_k$. 
For each $k\in\N$ and $w\in P_k$, we set $A_k(w)$ to be the unique point in
$P_{k-1}$ such that $w\in P_{k}(A_k(w))$  
with the property that $w\in B(A_k(w),2\eta^{k-1}r)\cap\partial\Om$. Moreover, we set
$D_k(w)=P_{k+1}\cap B(w,2\eta^kr)$. Note that $P_k\cap B(w,2\eta^kr)\subset D_k(w)$.
As the points of $P_{k}$ are from boundaries of well-placed balls of radius $\eta^{k}r$, their centers are at a mutual
distance of  
at least $8\eta^{k}r$. Thus the sets $D_{k}(w)$ with $w\in P_{k}$ form a pairwise-disjoint partition of $P_{k+1}$. 
\end{constr}

The above construction gives us a Cantor-type subset of $\partial\Om$, which we will see
in the next section to be a subset of $\partial\Om_{z_0}(c)\cap\partial\Om\cap B(z_0,3d_\Om(z_0))$ 
for a large-enough choice of $c$.
We now create chains of balls with controlled length connecting consecutive layers of our construction above
before constructing a Frostman measure associated with $P_\infty$.
	
\begin{lemma}\label{lem-chain-bound-M}
There is a positive integer $M>3$, depending solely on the doubling constant of $\mu$ and the scale $\eta$, 
such that the following 
holds. Let $k$ be a positive integer and $B\in \mathcal{W}_k$, and let $w\in P_{k-1}$ such that $B\in \mathcal{W}_k(w)$. 
Let $B(z_w,\eta^{k-1}r)\in\mathcal{W}_{k-1}$ be the ball such that $w\in\partial B(z_w,\eta^{k-1}r)\cap\partial\Om$. 
Then there exists a chain of balls $B=:B_0^*, B_1^*,\cdots, B_N^*:=B(z_w,\eta^{k}r)$, all of radii $\eta^kr$, with 
$N\le M$ and  $B_i^*\subset B(w, 2\eta^{k-1}r)$ for $i=0,\ldots,N$.
\end{lemma}

\begin{proof}
By construction, $B$ is chainable to $B(z_w,\eta^{k}r)$ within 
$\Om\cap B(w,2\eta^{k-1}r)$ 
by some chain $B_0^*,\dots, B_N^*$,
as in the statement of the lemma, with radii $\eta^kr$.
Choose the chain with the smallest possible number of balls $N$ in the chain. We will now show that $N$ can be bounded by a 
constant $M$ that depends solely on $\eta$ and $\mu$. Let $K$ be the maximal number of points in $B(w_k,2\eta^{k-1}r)$ 
that is $\eta^{k}r/8$-separated, that is,
with pairwise distance at least $\eta^{k} r/8$. Set $S$ to be the collection of these points. By the doubling property of $\mu$,
there is at most $M/2$ number of such points, that is, $K\le M/2$, with $M$ depending solely on the doubling constant 
and $\eta^{k-1}r/(\eta^{k}r/8)=8/\eta$. We now show that $N\le M$.

For each ball $B_i^*$ with center $x_{B_i^*}$, we choose a point
$s_i\in S\cap B(x_{B_i^*},\eta^{k} r/8)$.
Such a point $s_i$ must 
exist by the maximality of the collection $S$. If $N>M=2K$, then by the pigeonhole principle there must 
exist some $s\in S$ and three indices $i_1,i_2,i_3$ (in increasing order) such that $s_{i_1}=s_{i_2}=s_{i_3}=s$.
We have $|i_3-i_1|\geq 2$ since $i_2$ lies between them. Notice that 
\[
d(x_{B_{i_3}^*}, x_{B_{i_1}^*})\leq 
d(x_{B_{i_3}^*}, s) + d(s,x_{B_{i_1}^*})\leq \eta^{k} r/4,
\]
and thus $x_{B_{i_3}^*}\in \overline{B_{i_1}^*}$. 
Therefore, if we remove the balls 
$B_j^*$ for $j=i_{1}+1, \dots, i_3-1$ from the chain $B_0^*, \dots, B_N^*$, the balls that remain form a strictly shorter chain. This is a contradiction 
of the minimality of $N$ and thus $N\le M$ as claimed.
\end{proof}

\begin{prop}\label{prop:Frostman}
Under the hypotheses stated in Theorem~\ref{thm:main}, and with $P_\infty$ as constructed above, we have
\begin{equation}\label{eqn:Pinfty-lowerbound}
\frac{\mu(B(w_0,r))}{r^p}\lesssim \Hcali{r}^{-p}(P_\infty).
\end{equation}
Moreover, there is a Radon probability measure $\nu_F$, supported on $P_\infty$, such that 
for $\xi\in P_\infty$ and $0<\rho<r$,
\begin{equation}\label{eqn:frostmanestimate}
\nu_F(B(\xi,\rho))\lesssim \frac{\mu(B(\xi,\rho))}{\mu(B(w_0,r))}\,  \left(\frac{r}{\rho}\right)^{p}.
\end{equation}
\end{prop}
	
\begin{proof}
We prove the proposition by constructing a Frostman measure on the set $P_\infty$, that is, a measure $\nu_F$ on $P_\infty$ 
such that for $\xi\in P_\infty$ and $0<\rho<r$, the estimate~\eqref{eqn:frostmanestimate} holds.
This would complete the proof because  
for any cover of $P_\infty$ by balls $\{B_i(\rho_i)\}_{i\in I\subset\N}$ centered at points in $P_\infty$
with $\rho_i<r$, we have that
\[
1=\nu_F(P_\infty)\le \sum_i\nu_F(B_i(\rho_i))\lesssim \frac{r^{p}}{\mu(B(w_0,r))}\sum_i\frac{\mu(B_i(\rho_i))}{\rho_i^{p}}.
\]
Thus we obtain, by taking the infimum of the rightmost-hand side of the above over all covers, 
exactly the desired estimate \eqref{eqn:Pinfty-lowerbound}. 
Inequality~\eqref{eqn:frostmanestimate} will be established towards the end of the proof of this proposition,
see~\eqref{eq:Frostman-Control} below.

To construct a Frostman measure supported on $P_\infty$, we inductively construct the function 
$a:\bigcup_{k=0}^\infty P_k\times\{k\}\to[0,1]$ as follows. We first set
$a(w_0,0):=1$, and for $w\in P_1$, set
\[
a(w,1):=\frac{\mu(B(w,\eta r))}{\sum_{\zeta\in D_0(w_0)}\mu(B(\zeta,\eta r))}.
\]
Note that $P_0=\{w_0\}$; hence $D_0(w_0)=P_1$. It follows that $\sum_{w\in P_1}a(w,1)=1$.
Once $a(w,k)$ have been determined for each $w\in P_k$, we next set for $w\in P_{k+1}$,
\[
a(w,k+1):=a(A_{k+1}(w),k)\, \frac{\mu(B(w,\eta^{k+1}r))}{\sum_{\zeta\in D_k(A_{k+1}(w))}\mu(B(\zeta,\eta^{k+1}r))}.
\]
As pointed out in the construction of the sets $P_k$, we know that $D_k(w)$, $w\in P_k$ forms a pairwise
disjoint partition of $P_{k+1}$; we have for each $w'\in P_k$ that
\[
\sum_{w\in D_k(w')}a(w,k+1)=a(w',k).
\]
Finally, for each non-negative integer $k$, we set the measure
\begin{equation}\label{eq:Frostman-step-k}
\nu_k:=\sum_{w\in P_k}a(w,k)\, \delta_{w},
\end{equation}
where, for each $x\in X$, the measure $\delta_x$ is the probability measure supported on the singleton set $\{x\}$.
Note that the support of $\nu_k$ is $P_k\subset P_\infty$ and that $\nu_k(X)=1$. Indeed, it is immediate that $\nu_0(X)=1$, and that
\[
\nu_k(X)=\sum_{w\in P_k}a(w,k).
\]
By construction, when $k\ge 1$ we know that for each $\zeta\in P_{k-1}$,
\[
\sum_{w\in D_{k-1}(\zeta)}a(w,k)=a(\zeta,k-1),
\]
and so by induction we see that $\nu_k(X)=1$ for each $k\ge 1$ as well.
 
In the rest of this argument, we use the fact that necessarily, $\eta<\eta_1<\tfrac{1}{168}$.
Suppose that $w\in D_{k+l}(D_{k+l-1}(\cdots(D_k(w_k))\cdots))$
for some $w_k\in P_k$ and $l\ge 1$. Then we must have that $w\in B(w_k,3\eta^kr)$ and,
by the well-placedness property, the ball $B(w_k,4\eta^kr)$ does not intersect the ball $B(z_k,4\eta^kr)$
whenever
$w_k,z_k\in P_k$ with $w_k\ne z_k$.
Therefore, if $k_2>k_1$ and $w\in P_{k_1}$, then $w\in P_{k_2}$ and
\begin{equation}\label{eq:nontriv-Frost}
\nu_{k_2}(B(w,3\eta^{k_1}r))=\nu_{k_1}(B(w,\eta^{k_1}r))=a(w,k_1).
\end{equation}
Note also that for $k\ge 1$ and $w\in P_k$, by Lemma~\ref{lem:not-counting} 
with $A_k(w)$ playing the role of $w$, $x_{A_k(w)}$ playing the role of $z$, and with
$\eta_0=\eta$ and $r_0=\eta^{k-1}r$, we have 
\begin{align*}
\eta^q\, \mu(B(x_{A_k(w)}, \eta^{k-1}r))&\le K_0\, \sum_{\zeta\in D_k(A_k(w))\cap B(A_k(w),\eta^{k-1}r)}\mu(B(\zeta,\eta^{k}r))\\
 &\le K_0\, \sum_{\zeta\in D_k(A_{k}(w))}\mu(B(\zeta,\eta^{k}r)),
\end{align*}
and so
\begin{align*}
a(w,k)=\nu_k(B(w,\eta^kr))&=a(A_k(w),k-1)\frac{\mu(B(w,\eta^{k}r))}{\sum_{\zeta\in D_k(A_{k}(w))}\mu(B(\zeta,\eta^{k}r))}\\
&\le a(A_k(w),k-1)\frac{\mu(B(w,\eta^{k}r))}{K_0^{-1}\eta^q\mu(B(A_k(w),\eta^{k-1}r))}.
\end{align*}
An inductive argument gives us that for $k\ge 1$,
\[
a(w,k)\le \frac{\mu(B(w,\eta^kr))}{(K_0^{-1}\eta^q)^k\, \mu(B(w_0,r))}.
\]
Note that the constant $K_0$ does not depend on $\eta$, nor on $w$.
We fix $\eps>0$.
So if we choose $\eta_1$ small enough so that in addition we have $K_0^{-1}\ge \eta_1^{\eps}$, 
then for $0<\eta\le \eta_1$ we have that 
\[
a(w,k)\le \frac{\mu(B(w,\eta^kr))}{\eta^{(q+\eps)k}\, \mu(B(w_0,r))}.
\]
Let $0<\rho<r$ and $\xi\in\partial\Om$ such that $B(\xi,\rho)\cap P_\infty$ is non-empty. 
If $\rho\ge r/2$, then the claim follows from the doubling property of $\mu$ and the 
fact that $\nu_F(X)\le 1$ once $\nu_F$ is constructed, because each $\nu_k$ is a probability measure. Hence
we will focus on the case that $\rho<r/2$. Let $k_\rho$ be the unique non-negative integer such that
\[
\frac12\eta^{k_\rho+1}r\le \rho<\frac12\eta^{k_\rho}r.
\]
Then, as $P_k\subset P_{k+1}$ for each non-negative integer $k$, it follows that there is some positive integer $j\ge k_\rho$ and
$w\in P_j$ such that $w\in B(\xi,\rho)$.
 
By induction, we see that for each non-negative integer $m\le j$ there is some $w_m\in P_m$
such that $d(w,w_m)<\tfrac{2}{1-\eta}\eta^mr$. It follows that 
\[
d(\xi,w_m)<\frac{2}{1-\eta}\eta^mr+\rho.
\]
For each $w'\in P_m$ with $w'\ne w_m$, we have $d(w',w_m)\ge 8\eta^mr$ by the well-placedness property.
 For the choice of $m=k_\rho$, we have that
\[
B(\xi,\rho)\subset B(w_m,\tfrac{2}{1-\eta}\eta^mr+2\rho),\ \text{ and }\ B(w_m,\tfrac{2}{1-\eta}\eta^mr+2\rho)\cap P_m=\{w_m\}.
\]
So, thanks to~\eqref{eq:nontriv-Frost} and by the fact that $\tfrac{2}{1-\eta}\eta^{k_\rho}r+2\rho\le 3\eta^{k_\rho}r$,
 we have for each $k\ge k_\rho+1$,
\begin{align*}
\nu_k(B(\xi,\rho))\le \nu_k(B(w_{k_\rho},\tfrac{2}{1-\eta}\eta^{k_\rho}r+2\rho)) 
	&= \nu_{k_\rho}(B(w_{k_\rho},\tfrac{2}{1-\eta}\eta^{k_\rho}r+2\rho))\\
	&\le\nu_{k_\rho}(B(w_{k_\rho},\eta^{k_\rho}r))\\
	&\le \frac{\mu(B(w_{k_\rho},\eta^{k_\rho}r))}{\eta^{k_\rho(q+\eps)}\, \mu(B(w_0,r))}\\
	&\approx \frac{\mu(B(\xi,\rho))}{\mu(B(w_0,r))}\,  \left(\frac{r}{\rho}\right)^{q+\eps}.
\end{align*}
Here the comparison constant depends solely on the doubling constant of $\mu$ and the choice of  $\eta$.
		
By the Banach--Alaoglu theorem, we have a weak*-convergent 
subsequence of the sequence of Radon probability measures $\nu_k$
to a Radon measure $\nu_F$ that lives inside $\partial\Om$, with $\nu_F(P_\infty)=1$. We obtain then 
also that for $\xi\in\partial\Om$ and $0<\rho<r$,
\begin{equation}\label{eq:Frostman-Control}
\nu_F(B(\xi,\rho))\lesssim \frac{\mu(B(\xi,\rho))}{\mu(B(w_0,r))}\,  \left(\frac{r}{\rho}\right)^{q+\eps}.
\end{equation}
This measure $\nu_F$ is our Frostman measure, with the choice of $\eps=p-q$.
\end{proof}

\section{Proof of Theorem~\ref{thm:main}}\label{Sec:4}  
	
In this section, we complete the proof of the first main result of this paper, Theorem~\ref{thm:main}. 
Thanks to Proposition~\ref{prop:Frostman} and the fact that the visible boundary is closed, it suffices to show that 
$\bigcup_k P_k$ is a subset of $\partial\Omega_{z_0}(c) \cap \partial \Omega$ for some $c\geq 1$. See 
Construction~\ref{def:P-infty} for the choice of the sets $P_k$.

Take $w\in P_k$ for some $k\geq 0$. Then there is a point $z_w$ such that $w\in\partial B(z_w,\eta^kr)$
and $B(z_w,\eta^kr)\in \mathcal{W}_k$. 
Joining $w$ to $z_w$ by a geodesic $\widetilde\gamma$, we have that for any $z$ along $\widetilde\gamma$, 
\begin{equation}\label{eq:John-Ball}
d_\Om(z)=d_B(z)\geq d_B(z_w)-d(z_w,z)=\ell(\widetilde\gamma)-\ell(\widetilde\gamma_{z_w,z})= \ell(\widetilde\gamma_{z,w}).
\end{equation}
A similar argument to the above also tells us that if $B(x,\rho)\subset\Om$, $y\in B(x,\rho)$, and $\beta$ is a geodesic
with end points $x,y$, then $\beta\subset \Om$ and for each point $z$ in $\beta$,
\[
d_\Om(z)\ge d_\Om(x)-d(x,z)\ge d(x,y)-d(x,z)=d(y,z). 
\]
By construction, $B(z_w,\eta^kr)\in \mathcal{W}_k(w_{-1})$ for some $w_{-1}\in P_{k-1}$. Let $z_{w_{-1}}$ be the corresponding
point in $\Om$ such that $B(z_{w_{-1}},\eta^{k-1}r)\in \mathcal{W}_{k-1}$ and 
$w_{-1}\in \partial\Om\cap \partial B(z_{w_{-1}},\eta^{k-1}r)$. We denote by $\{B_i^k=B(x_{i,k},\eta^kr)\}_{i=1}^{N_k}$ the 
finite chain of balls joining $B(z_w,\eta^kr)$ to $B(z_{w_{-1}},\eta^kr)$ given by Lemma \ref{lem-chain-bound-M} of length $N_k\le M$ for some $M$ depending on the doubling constant and $\eta$. For $i=1,\cdots, N_k$ let $\gamma_{i,k}$ be a geodesic
with end points $x_{i,k}$, $x_{i+1,k}$, and note that $\gamma_{i,k}$ lies in $\Om$ with $\ell(\gamma_{i,k})< \eta^kr$.

As $d(x_{i,k},x_{i+1,k})<\eta^kr$, it follows that for each $z\in\gamma_{i,k}$, we either have $d(z,x_{i,k})\le \eta^kr/2$ or 
$d(z,x_{i+1,k})<\eta^kr/2$. Thus, $d_\Om(z)\ge \frac{\eta^k\,r}{2}$. Let $\gamma_k$ be the concatenation of
the curves $\gamma_{i,k}$, $i=1,\cdots, N_k$; then, $\ell(\gamma_k)\le N_k\eta^kr$, and so for each $z\in\gamma_k$ we have
\[
d_\Om(z)\ge \frac{\eta^kr}{2}\ge \frac{1}{2N_k}\ell(\gamma_k)\ge \frac{1}{2M}\ell(\gamma_k).
\]
Finally let $\gamma$ be the concatenation of $\widetilde{\gamma}$ and $\gamma_k,\, \gamma_{k-1},\, \cdots, \gamma_1$
to obtain a curve that connects $w$ to $z_0$ inside $\Om$. We claim that $\gamma$ is a John curve. Indeed, if 
$z\in\widetilde{\gamma}$, then by~\eqref{eq:John-Ball} we have that 
$d_\Om(z)\ge \ell(\gamma_{wz})=\ell(\widetilde{\gamma}_{w,z})$. If $z\in\gamma_j$ for some $j\in\{1,\cdots, k\}$, then
\[
d_\Om(z)\ge \frac{\eta^jr}{2}, \qquad 
\ell(\gamma_{w,z})\le \ell(\widetilde{\gamma})+\sum_{m=j}^k\ell(\gamma_m)\le \eta^kr+\sum_{m=j}^kM\, \eta^mr\le 2M\eta^jr.
\]
Thus, we have
\[
d_\Om(z)\ge \frac{1}{4M}\, \ell(\gamma_{w,z}).
\]
In other words, $\gamma$ is a $4M$-John curve with respect to $\Om$. 
	
Next, we show that, with $\gamma:[0,L]\to\overline{\Om}$,  the subdomain 
\[
C[\gamma]:=\bigcup_{t\in[0,L]}B(\gamma(t),d_{\Om}(\gamma(t))),
\]
is $(8M+1)$-John with respect to $C[\gamma]$ with John center $z_0$
and that $z_0\in C[\gamma]$. The latter is immediate from the fact that
$z_0$ is one of the terminal points of $\gamma$. We now show that $\gamma$ is a $4M$-John curve
with respect to this \emph{sub}domain $C[\gamma]$. Let $z\in\gamma$. Then by the above discussion we have that
\[
\ell(\gamma_{w,z})\le 4M\, d_\Om(z) \qquad \text{ and }\qquad 
d_\Om(z)\ge d_{C[\gamma]}(z)\ge d_{B(z,d_\Om(z))}(z)=d_\Om(z).
\]
The claim follows.
	
Finally, suppose that $y\in C[\gamma]$. Then there is some number $t_y\in[0,L]$ such that 
$y\in B(\gamma(t_y), d_\Om(\gamma(t_y)))\subset C[\gamma]$. Let $\beta_y$ be a geodesic connecting $y$ to
$\gamma(t_y)$, and let $\alpha_y$ be the concatenation of the paths $\beta_y$ with $\gamma_{\gamma(t_y),z_0}$.
We now show that $\alpha_y$ is an $(8M+1)$-John curve with respect to $C[\gamma]$. To this end, let 
$z\in \alpha_y$. If $z\in\beta_y$, then an argument as in~\eqref{eq:John-Ball} gives
\[
d_{C[\gamma]}(z)\ge d_{B(\gamma(t_y), d_\Om(\gamma(t_y)))}(z)\ge \ell((\beta_y)_{z,y}).
\]
If $z\in\gamma_{\gamma(t_y),z_0}$, then by the first part of the argument above, we have
\[
d_\Om(z)=d_{C[\gamma]}(z)\ge \frac{\ell(\gamma_{\gamma(t_y),z})}{4M}.
\]
Moreover, by the above and the triangle inequality, we have that
\[
\ell(\beta_y)\le d_\Om(\gamma(t_y))\le d(z,\gamma(t_y))+d_\Om(z)\le (4M+1)d_\Om(z)=(4M+1)d_{C[\gamma]}(z).
\]
It follows that 
\[
\ell((\alpha_y)_{z,y})=\ell(\gamma_{\gamma(t_y),z})+\ell(\beta_y)\le (8M+1)d_{C[\gamma]}(z).
\]
From the above two cases we have that $\alpha_y$ is an $(8M+1)$-John curve with respect to $C[\gamma]$ with
John center $z_0$. It follows that $C[\gamma]$ is $(8M+1)$-John with John center $z_0$.
	
Since unions of $(8M+1)$-John subdomains of $\Om$ with John center $z_0$ is also $(8M+1)$-John with John
center $z_0$, it follows that $C[\gamma]\subset\Om_{z_0}(8M+1)$, and as 
$w\in\partial C[\gamma]\cap\partial\Om$, it follows that 
$\bigcup_k P_k\subset\partial\Om_{z_0}(8M+1)\cap\partial\Om$. An appeal to Proposition~\ref{prop:Frostman}
now completes the proof of Theorem~\ref{thm:main}.

\section{Traces of Newton-Sobolev functions on the visible boundary} \label{Sec:5} 
	
	In this section we provide a proof of the main trace result, Theorem~\ref{thm:main2}.
Indeed, we will prove a more general version of Theorem~\ref{thm:main2}, namely, Theorem~\ref{thm:subsidiary1}
below. In what follows, we denote by $F_\nu$ a compact subset of 
$F=B(z_0,3\,d_\Om(z_0))\cap\partial\Om_{z_0}(c)\cap\partial\Om$
and a Radon measure $\nu$ supported on $F_\nu$ for which there is a constant $c_2>0$ such that 
whenever $\zeta\in F_\nu$ and $0<r<3\,d_{\Om}(z_0)$, we have
\begin{equation}\label{eq:upper-bound-local}
\nu(B(\zeta,r))\le c_2\, \frac{\mu(B(\zeta,r)\cap\Om)}{r^p}.
\end{equation}
Because of the metric doubling property of $\overline\Om$, for each 
$L\geq 3$, we may assume that the above inequality holds for every $0<r<L\,d_\Om(z_0)$ at 
the expense of increasing $c_2$.

With this additional condition satisfied, we are able to gain control over traces on $F_\nu$; this is the 
content of Theorem~\ref{thm:main2} below. Note that the set $P_\infty$ satisfies this condition,
but it is possible for there to be a larger subset $F_\nu\subset F$ satisfying this condition as well. The set
$P_\infty$ is a Cantor-type set because of the well-placedness condition from Definition~\ref{def:well-placed}.
It is possible to have a larger set $F_\nu$, which might contain non-trivial continua while satisfying the above
dominance condition.

\begin{theorem}\label{thm:subsidiary1}
Suppose that $(X,d,\mu)$ is complete. 
Let
the domain $\Om\subset X$, $z_0\in\Om$, and $c>1$ be such that
$F=B(z_0,3d_\Om(z_0))\cap\partial\Om\cap\partial\Om_{z_0}(c)$ contains a compact subset $F_\nu$
and a Radon measure $\nu$ supported on $F_\nu$ satisfying~\eqref{eq:upper-bound-local}.
In addition, suppose also that $\mu\vert_\Om$ is doubling and supports a $\widehat{q}$-Poincar\'e inequality
for some $\max\{1,p\}<\widehat{q}<q<\infty$.
Then there is a bounded linear trace operator 
\[
T:N^{1,q}(\Om)\to B_{q,q,2}^{1-p/q}(F_\nu,d,\nu)\cap L^q(F_\nu,\nu)
\]
such that, with $D(z_0):=B(z_0,10c\,d_\Om(z_0))\cap\Om$, we have
\begin{equation}\label{eq:trace-energy}
\|Tu\|_{B^{1-p/q}_{q,q,2}(F_\nu,d,\nu)}^q\lesssim \inf_{g_u}\int_{D(z_0)} {g_u}^q\, d\mu,
\end{equation}
where the infimum is over all upper 
gradients $g_u$ of $u$ and the implied constant is independent of $u$ and $z_0$,
and 
\begin{equation}\label{eq:Lp-estimates}
\|Tu\|_{L^q(F_\nu,\nu)}^q\lesssim d_{\Om}(z_0)^{-p}\|u\|_{L^{q}(\Om_{z_0}(c))}^q 
   + d_{\Om}(z_0)^{q-p}\inf_{g_u}\int_{D(z_0)} {g_u}^q\, d\mu,
\end{equation}
where the infimum is over all upper 
gradients $g_u$ of $u$ and the implied constant is independent of $u$.
\end{theorem}

In the setting considered in Theorem~\ref{thm:subsidiary1}, 
we know from \cite[Proposition~7.1.]{Ash}
that $N^{1,q}(\Om)=N^{1,q}(\overline{\Om})$ and that the complete space $\overline{\Om}$, equipped with the zero-extension
of $\mu\vert_\Om$, is also doubling and supports a $\widehat{q}$-Poincar\'e inequality.
By the assumption that $\nu$ is supported on $F_\nu$, we know that
$\nu(B(\zeta,\rho)\cap F_\nu)>0$ for each $\zeta\in F_\nu$
and for all $\rho>0$. 
	
\begin{lemma}\label{lem:LebesguePts}
Under the assumptions of Theorem~\ref{thm:subsidiary1}, 
for each $u\in N^{1,q}(\Om)$ we have that $\nu$-a.e.~$x\in F_\nu$ is a Lebesgue point of $u$.
\end{lemma}

Note that in this lemma we do not need the full strength of the above upper bound inequality~\eqref{eq:upper-bound-local}; that inequality
is needed only in the proof of the trace theorem, Theorem~\ref{thm:main2}. Instead, we need that $\nu \ll \mathcal{H}^{-p}$, which is implied by \eqref{eq:upper-bound-local}.

\begin{proof}
It follows from the discussion preceding the statement of the lemma, together with~\cite[Proposition~3.11]{GKS} and~\cite[Theorem~9.2.8]{HKSTbook}, that
	$\nu$-a.e.~$x\in F_\nu$ is a Lebesgue point of $u$.
\end{proof}

We are now ready to
prove Theorem~\ref{thm:subsidiary1}. This proof is broken into two parts, with
the first part proving the energy estimate~\eqref{eq:trace-energy}, and the second part proving the 
$L^p$-estimates~\eqref{eq:Lp-estimates} claimed in the theorem.
In proving this, we take inspiration from the proof of~\cite[Theorem~5.6]{Mal}. 
Here, for $z\in F_\nu$ a Lebesgue point of $u\in N^{1,q}(\overline{\Om})$, the trace $Tu(z)$ is defined by the property that
\begin{equation}\label{eq:Thrace}
\limsup_{r\to 0^+}\jint_{B(z,r)\cap\Om}|u-Tu(z)|\, d\mu=0.
\end{equation}
We begin by showing that the Besov energy of the trace $Tu$ on $F_\nu$ can be controlled by the 
Sobolev energy of $u\in N^{1,q}(\Om)$.	

\begin{proof}[Proof of Theorem~\ref{thm:subsidiary1}: Energy estimates]
Fix $u\in N^{1,q}(\Om)$ and $g_u\in L^q(\Om)$ an upper gradient of $u$,
and consider two Lebesgue points $y,z\in F_\nu\subset\partial\Om_{z_0}(c)$
of $u$. Since $F_\nu \subset F$, we can find two $c$-John (in $\Om_{z_0}(c)$)
curves $\gamma_y$ and 
$\gamma_z$, parametrized by arclength, starting from $y$ and $z$, respectively, and both 
terminating at the John center $z_0$. For each $k\in\N$, we set
\begin{equation}\label{eq:tk-rk}
t_k=t_{-k}:=\left(1-\frac{1}{2c}\right)^k\, d(y,z), \qquad r_k=r_{-k}:=\frac{1}{2c}\, t_k.
\end{equation}
Notice that the John constant $c$ is larger than $1$ as in the statement of the theorem. 
We next set $B_0:=B(z,3d(y,z))$, and for each $k\in\N$ we set
$B_k:=B(\gamma_z(t_k),r_k)$, $B_{-k}:=B(\gamma_y(t_{-k}),r_{-k})$. 
Note that for $k\in\N$, 
$|t_{k+1}-t_k|=d(y,z)\, \left(1-\tfrac{1}{2c}\right)^k\, \tfrac{1}{2c}=r_k$; therefore, for each $k\in\N$ we have that
$\gamma_{y}(t_{k+1})\in \overline{B_{-k}}$ and $\gamma_z(t_{k+1})\in \overline{B_k}$. Moreover, 
$r_k+r_{k+1}=\tfrac{1}{2c}\, d(y,z)\l \left(1-\tfrac{1}{2c}\right)^k\left[2-\tfrac{1}{2c}\right]$; therefore, setting
$\alpha=2-\tfrac{1}{2c}$, 
we have that $B_{k+1}\subset \alpha B_k$ 
and $B_{-(k+1)}\subset\alpha B_{-k}$
whenever $k\in\N\cup\{0\}$.
Also, from the $c$-John condition we know that when $k\in\N$, we have
$d_\Om(\gamma_z(t_k))\ge t_k/c=2r_k$; hence, as
$\alpha<2$, we have that $\alpha B_k\subset\Om$. A similar argument gives $\alpha B_{-k}\subset\Om$.

We now estimate $Tu(z)-Tu(y)$ as follows:
\begin{align*}
	|Tu(z)-Tu(y)|\le \sum_{k\in\Z}|u_{B_k}-u_{B_{k+1}}|
	&\lesssim\sum_{k\in\Z}\jint_{\alpha B_k}|u-u_{\alpha B_k}|\, d\mu\\
	&\lesssim \sum_{k\in\Z}r_k\, \left(\jint_{\alpha B_k}g_u^{\widehat{q}}\, d\mu\right)^{1/\widehat{q}},
\end{align*}
where we also set $r_0:=3d(y,z)$.
Note that $r_1=\tfrac{1}{2c}\, \left(1-\tfrac{1}{2c}\right)\, d(y,z)< r_0$;
moreover, for each $w\in \alpha B_1$ we have that 
\begin{align*}
d(w,z)\le d(w,\gamma_z(t_1))+t_1<\alpha r_1+t_1&= \left(\frac{\alpha}{2c}+1\right)\, \left(1-\frac{1}{2c}\right)\, d(y,z)\\
  &=\left(\frac{1}{c}-\frac{1}{4c^2}+1\right)\, \left(1-\frac{1}{2c}\right)\, d(y,z)< 2\, d(y,z).
\end{align*}
It follows that $\alpha B_1\subset B_0\subset\alpha B_0$. Hence, for each $w\in\alpha B_1$ we have that
\[
M_\Om (\chi_{D(z_0)}\, g_u^{\widehat{q}})(w)\ge \jint_{\alpha B_0}g_u^{\widehat{q}}\, d\mu,
\]
where, for each $w\in\Om$, we set
\[
M_\Om (\chi_{D(z_0)}\, g_u^{\widehat{q}})(w):=\sup_{w\in B}\jint_{B\cap\Om}\chi_{D(z_0)}\, g_u^{\widehat{q}}\, d\mu
\]
with the supremum taken over all balls $B$ centered at points in $\overline{\Om}$ such that $w\in B$.
So, setting $g:=\left(g_u^{\widehat{q}}+M_\Om (\chi_{D(z_0)}\, g_u^{\widehat{q}})\right)^{1/{\widehat{q}}}$, we have
\[
r_0\left(\jint_{\alpha B_0}g^{\widehat{q}}\, d\mu\right)^{\frac{1}{\widehat{q}}}
	\lesssim r_1\, \left(\jint_{\alpha B_1}g^{\widehat{q}}\, d\mu\right)^{\frac{1}{\widehat{q}}}.
\]
Denoting the set of all non-zero integers by $\Z^*$, we have
\[
|Tu(z)-Tu(y)|\lesssim \sum_{k\in\Z}r_k\, \left(\jint_{\alpha B_k}g_u^{\widehat{q}}\, d\mu\right)^{\frac{1}{\widehat{q}}}
	\le \sum_{k\in\Z^*}r_k\, \left(\jint_{\alpha B_k}g^{q}\, d\mu\right)^{\frac{1}{q}}.
\]
Recall that $p<q$; hence we can choose $\eps>0$ such that $p+\eps<q$. Now, by an application of 
H\"older's inequality, we see that
\begin{align*}
|Tu(y)-Tu(z)|\lesssim 
	\sum_{k\in\Z^*}r_k^{1-\frac{p+\eps}{q}}\, r_k^{\frac{p+\eps}{q}}\left(\fint_{\alpha B_k}\!g^{{q}}\,d\mu\right)^{\frac{1}{{q}}}
	\leq \left(\sum_{k\in\Z^*}r_k^{q'\left(1-\frac{p+\eps}{q}\right)}\right)^{\frac{1}{{q'}}}
	\left(\sum_{k\in\Z^*}r_k^{p+\eps}\fint_{\alpha B_k}\!g^{{q}}\,d\mu\right)^{\frac{1}{{q}}},
\end{align*}
where $q'=\frac{q}{q-1}$ is the H\"older dual of $q$.
		
Estimating the first factor on the right-hand side of the above inequality, we find that
\begin{align*}
\sum_{k\in\Z^*}r_k^{q'\left(1-\frac{p+\eps}{q}\right)} = \sum_{k\in\Z^*}r_k^{\frac{q-(p+\eps)}{q-1}}
	&=\sum_{k\in\Z^*}\left(\frac{d(y,z)}{2c}\left(1-\frac{1}{2c}\right)^{|k|}\right)^{\frac{q-(p+\eps)}{q-1}}\\
	&\approx d(y,z)^{\frac{q-(p+\eps)}{q-1}}\sum_{k\in\Z^*}\left(1-\frac{1}{2c}\right)^{|k|\frac{q-(p+\eps)}{q-1}}.
\end{align*}
The choice of $\eps$ such that $p+\eps<q$ ensures that
\[
\left(1-\frac{1}{2c}\right)^{\frac{q-(p+\eps)}{q-1}}<1,
\]
and so we have that 
\[
\sum_{k\in\Z^*}r_k^{q'\left(1-\frac{p+\eps}{q}\right)}\lesssim d(y,z)^{\frac{q-(p+\eps)}{q-1}}.
\]
Hence,
\[
|Tu(y)-Tu(z)|\lesssim d(y,z)^{\tfrac{q-(p+\eps)}{q}}\, 
	\left(\sum_{k\in\Z^*}r_k^{p+\eps}\fint_{\alpha B_k}\!g^{{q}}\,d\mu\right)^{\frac{1}{q}}.
\]
		
We set $w_k=z$ for $k\ge 1$ and $w_k=y$ for $k\le -1$; then for each $k\in\Z^*$, by the doubling property
of $\mu\vert_\Om$ and by $B(w_k,r_k)\cap\Om\subset (2c+1)B_k\cap\Om$, we have that 
\begin{equation}\label{eq:tau}
\frac{r_k^p}{\mu(\alpha B_k)}\lesssim\frac{r_k^p}{\mu(B(w_k,(4\tau)^2\, r_k)\cap\Om)}\lesssim\frac{1}{\nu(B(w_k,(4\tau)^2 r_k))},
\end{equation}
where we have chosen $\tau=8c^2$. 
Here we have also used the upper bound estimate~\eqref{eq:upper-bound-local}. Now 
\[
\frac{|Tu(y)-Tu(z)|}{d(y,z)^{1-p/q}}\lesssim 
d(y,z)^{-\eps/q}\, \left(\sum_{k\in \Z^*}\frac{r_k^\eps}{\nu(B(w_k,(4\tau)^2 r_k))}\int_{\alpha B_k}g^q\, d\mu\right)^{\frac{1}{q}},
\]
that is,
\[
\frac{|Tu(y)-Tu(z)|^q}{d(y,z)^{q-p}\, \nu(B(z,2d(y,z)))} \lesssim
\frac{d(y,z)^{-\eps}}{\nu(B(z,2d(y,z)))}\, \sum_{k\in \Z^*}\frac{r_k^\eps}{\nu(B(w_k, (4\tau)^2\, r_k))}\int_{\alpha B_k}g^q\, d\mu.
\]
We denote 
\[
C_{y,z}:=\bigcup_{k\geq 1}\alpha B_k\qquad C_{z,y}:=\bigcup_{k\leq -1}\alpha B_{k}
\]
the two ``cones'', $C_{y,z}$ with vertex at $z$ and $C_{z,y}$ with vertex at $y$.
Note that these cones consist of balls with centers located on John curves terminating at the John 
center $z_0$; hence, $C_{y,z}\cup C_{z,y}\subset D(z_0)$.
For $w\in\alpha B_k$ we have that 
$d(w,w_k)<\alpha r_k+t_k=(\alpha+2c)\,r_k<4\tau r_k$. Hence,
\begin{align*}
\sum_{k\in \Z^*}\frac{r_k^\eps}{\nu(B(w_k,(4\tau)^2\, r_k))}&\int_{\alpha B_k}g^q\, d\mu
\lesssim
\sum_{k\in \Z^*}\int_{\alpha B_k}\frac{d(w,w_k)^\eps\, g(w)^q}{\nu(B(w_k, 4\tau d(w,w_k)))}\, d\mu(w)\\
\approx 
         \sum_{k=1}^\infty \int_{\alpha B_k}&\frac{d(w,z)^\eps\, g(w)^q}{\nu(B(z,4\tau d(w,z)))}\, d\mu(w)
	+\sum_{k=1}^\infty \int_{\alpha B_{-k}}\frac{d(w,y)^\eps\, g(w)^q}{\nu(B(y,4\tau d(w,y)))}\, d\mu(w)\\
	\approx \int_{C_{y,z}}&\frac{d(w,z)^\eps\, g(w)^q}{\nu(B(z,4\tau d(w,z)))}\, d\mu(w)
	+\int_{C_{z,y}}\frac{d(w,y)^\eps\, g(w)^q}{\nu(B(y,4\tau d(w,y)))}\, d\mu(w).
\end{align*}
In the last step above we have used the fact that the balls $\alpha B_k$, $k\in\Z^*$, have bounded overlap; no more
than $3$ of these balls overlap at any point in $\Om$. 
Now integrating over $F_\nu\times F_\nu$, we obtain
\begin{align*}
\int_{F_\nu}\int_{F_\nu} &\frac{|Tu(y)-Tu(z)|^q}{d(y,z)^{q-p}\, \nu(B(z,2d(y,z)))}\,  d\nu(y)\, d\nu(z) \lesssim\\
\int_{F_\nu}\int_{F_\nu} & \frac{d(y,z)^{-\eps}}{\nu(B(z,2d(y,z)))}\,
      \bigg[\int_{C_{y,z}}\frac{d(w,z)^\eps\, g(w)^q}{\nu(B(z,4\tau d(w,z)))}\, d\mu(w)\\
&\hskip 3cm+\int_{C_{z,y}}\frac{d(w,y)^\eps\, g(w)^q}{\nu(B(y,4\tau d(w,y)))}\, d\mu(w)\bigg]\, d\nu(y)\, d\nu(z)
	\ \ =\ \ I+J,
\end{align*}
where
\begin{align*}
I&:=\int_{F_\nu}\int_{F_\nu} \frac{d(y,z)^{-\eps}}{\nu(B(z,2d(y,z)))}\, 
	\int_{C_{y,z}}\frac{d(w,z)^\eps\, g(w)^q}{\nu(B(z,4\tau d(w,z)))}\, d\mu(w)\, d\nu(y)\, d\nu(z),\\
	J&:=\int_{F_\nu}\int_{F_\nu} \frac{d(y,z)^{-\eps}}{\nu(B(z,2d(y,z)))}\, 
	\int_{C_{z,y}}\frac{d(w,y)^\eps\, g(w)^q}{\nu(B(y,4\tau d(w,y)))}\, d\mu(w)\, d\nu(y)\, d\nu(z).
\end{align*}
We obtain estimates for $I$, and a {\it mutatis mutandis} argument gives similar estimates for $J$.

Recall our choices of $t_k$ and $r_k$ from~\eqref{eq:tk-rk}.		
For $w\in C_{y,z}$ we know that there is some $k\ge 1$ such that $d(w,\gamma_z(t_k))<\alpha r_k$,
and so 
\[
d(w,z)<\alpha r_k+t_k
=\left(1+\frac{\alpha}{2c}\right)t_k
=\left(1+\frac{\alpha}{2c}\right)\, \left(1-\frac{1}{2c}\right)^k\, d(y,z)
	\le  \left(1+\frac{1}{c}\right)\left(1-\frac{1}{2c}\right)\, d(y,z). 
\]
Setting $\beta=\left(1+\tfrac{1}{c}\right)\left(1-\tfrac{1}{2c}\right)$,
 it follows that necessarily $d(y,z)>d(z,w)/\beta$. 
Moreover, as above, we have that $d(w,z)<\left(1+\frac{\alpha}{2c}\right)t_k$, and by
the John property of the curve $\gamma_z$, we also have 
$d_\Om(w)> d_\Om(\gamma(t_k))-\alpha r_k\ge \tfrac{t_k}{c}-\tfrac{\alpha}{2c}t_k=\tfrac{1}{4c^2}\, t_k=\tfrac{r_k}{2c}$. Hence
$d(w,z)<2c(2c+\alpha)d_\Om(w)$. So
if $w\in\Om$ is fixed, then we must have 
$z\in  
B(w,2\tau d_\Om(w))$, where we recall that $\tau=8c^2$. It follows that
\begin{align*}
I&=\int_{F_\nu}\int_{F_\nu} \frac{d(y,z)^{-\eps}}{\nu(B(z,2d(y,z)))}\, 
\int\limits_{D(z_0)}\frac{d(w,z)^\eps\, g(w)^q\, \chi_{C_{y,z}}(w)}{\nu(B(z,4\tau d(w,z)))}\, d\mu(w)\, d\nu(y)\, d\nu(z)\\
&\le \int_{F_\nu}\int_{F_\nu} \frac{d(y,z)^{-\eps}}{\nu(B(z,2d(y,z)))}\, 
\int\limits_{D(z_0)}\frac{d(w,z)^\eps\, g(w)^q\, \chi_{F_\nu\setminus B(z,d(w,z)/\beta)}(y)\,\chi_{F_\nu\cap B(w, 2\tau d_\Om(w))}(z)}{\nu(B(z,4\tau d(w,z)))}\, d\mu(w)\, d\nu(y)\, d\nu(z)\\
&=\int\limits_{D(z_0)}\int_{F_\nu}\int_{F_\nu} 
\frac{d(y,z)^{-\eps}}{\nu(B(z,2d(y,z)))}\, \frac{d(w,z)^\eps\, g(w)^q\, \chi_{F_\nu\setminus B(z,d(w,z)/\beta)}(y)\,\chi_{F_\nu\cap B(w, 2\tau d_\Om(w))}(z)}{\nu(B(z,4\tau d(w,z)))}\, 
d\nu(y)\, d\nu(z)\, d\mu(w)\\
&=\int\limits_{D(z_0)} g(w)^q\!\!\! \!\!\int\limits_{F_\nu\cap B(w,2\tau d_\Om(w))}\!\!\!d(w,z)^\eps\, \!\!\int\limits_{F_\nu\setminus B(z,d(w,z)/\beta)}
\!\!\frac{d(y,z)^{-\eps}}{\nu(B(z,2d(y,z)))}\, \frac{1}{\nu(B(z,4\tau d(w,z)))}\, 
d\nu(y)\, d\nu(z)\, d\mu(w).
\end{align*}
Note that
\begin{align*}
\int_{F_\nu\setminus B(z,d(w,z)/\beta)}
&\frac{d(y,z)^{-\eps}}{\nu(B(z,2d(y,z)))}\, \frac{1}{\nu(B(z,4\tau d(w,z)))}\, d\nu(y)\\
	=\sum_{j=1}^\infty &\int_{F_\nu\cap B(z, 2^jd(w,z)/\beta)\setminus B(z,2^{j-1}d(w,z)/\beta)}
	\frac{d(y,z)^{-\eps}}{\nu(B(z,2d(y,z)))}\, \frac{1}{\nu(B(z,4\tau d(w,z)))}\, d\nu(y)\\
	\lesssim \sum_{j=1}^\infty &\int_{F_\nu\cap B(z, 2^jd(w,z)/\beta)\setminus B(z,2^{j-1}d(w,z)/\beta)}
	\frac{2^{-\eps j}\, d(w,z)^{-\eps}}{\nu(B(z,2^{j}d(w,z)/\beta))}\, \frac{1}{\nu(B(z,4\tau d(w,z)))}\, d\nu(y)\\
	&\lesssim \sum_{j=1}^{K_{w,z}}  2^{-\eps j}d(w,z)^{-\eps}\, \frac{\nu(B(z,2^jd(w,z)/\beta))}{\nu(B(z,2^{j}d(w,z)/\beta))}\, 
	\frac{1}{\nu(B(z,4\tau d(w,z)))}.
\end{align*}
Here $K_{w,z}$ is the largest integer $j$ for which the annulus 
$F_\nu\cap B(z, 2^jd(w,z)/\beta)\setminus B(z,2^{j-1}d(w,z)/\beta)$ is non-empty
(recall that $F_\nu$ is bounded).
Simplifying, we obtain
\[
\int_{F_\nu\setminus B(z,d(w,z)/\beta)}
\frac{d(y,z)^{-\eps}}{\nu(B(z,2d(y,z)))}\, \frac{1}{\nu(B(z,4\tau d(w,z)))}\, d\nu(y)
\lesssim \frac{d(w,z)^{-\eps}}{\nu(B(z,4\tau d(w,z)))}.
\]
Hence
\[
I\lesssim \int_\Om g(w)^q\, \int_{F_\nu\cap B(w,2\tau d_\Om(w))}d(w,z)^\eps\, \frac{d(w,z)^{-\eps}}{\nu(B(z,4\tau d(w,z)))}\, d\nu(z)\, d\mu(w).
\]
Next, note that  $B(w,2\tau d_\Om(w))\cap F_\nu\subset B(z,4\tau d(w,z))$. Therefore,
\begin{align*}
	\int_{F_\nu\cap B(w,2\tau d_\Om(w))}d(w,z)^\eps\, \frac{d(w,z)^{-\eps}}{\nu(B(w,2\tau d(w,z)))}\, d\nu(z)
	\le\frac{\nu(B(w,2\tau d_\Om(w)))}{\nu(B(w,2\tau d_\Om(w)))}=1.
\end{align*}
It follows that
\[
I\lesssim \int_{D(z_0)}\! g(w)^q\, d\mu(w).
\]
Similarly, we obtain the estimate $J\lesssim\int_{D(z_0)} g^q\, d\mu$. 
		
Now, the boundedness of the Hardy-Littlewood maximal operator on 
$L^{q/\widehat{q}}(\Om,\mu)$ (since $\tfrac{q}{\widehat{q}}>1$)
and the doubling property of $\mu\vert_\Om$
implies that 
\begin{align*}
\int_{D(z_0)}\! g^q\, d\mu
&=\int_{D(z_0)} \! \left(g_u^{\widehat{q}}+M_\Om(\chi_{D(z_0)}\, g_u^{\widehat{q}})\right)^{\frac{q}{\widehat{q}}} d\mu\\
&\lesssim \int_{D(z_0)}\!g_u^q \,d\mu + \int_{\Om}\! \left(M_\Om (\chi_{D(z_0)}\, g_u^{\widehat{q}})\right)^{\frac{q}{\widehat{q}}} d\mu\\
&\lesssim \int_{D(z_0)}\!g_u^q \,d\mu,
\end{align*}
and so we reach the desired conclusion that
\[
\Vert Tu\Vert_{B^{1-p/q}_{q,q,2}(F_\nu, d, \nu)}^q\lesssim \int_{D(z_0)} g_u^q\, d\mu.\qedhere
\]
\end{proof}

Now we prove~\eqref{eq:Lp-estimates},
thus completing the proof of Theorem~\ref{thm:subsidiary1}.
Note that in the following argument  
we do not need the stronger $\widehat{q}$-Poincar\'e inequality; the weaker
$q$-Poincar\'e inequality suffices.

\begin{proof}[Proof of Theorem~\ref{thm:subsidiary1}: $L^q$-estimates]
As before, consider a Lebesgue point $z\in F_\nu$ of $u\in N^{1,q}(\Om)$, an upper gradient $g_u\in L^q(\Om)$ of $u$, and a 
$c$-John curve $\gamma_{z}$, parametrized by arclength, starting from $z$ and terminating at $z_0$. Modifying slightly the 
construction in the first part of the proof, for $k\in\N$ we set
		\[
		t_k:=\left(1-\frac{1}{2c}\right)^k\, \ell(\gamma_z), \qquad r_k:=\frac{1}{2c}\, t_k, \qquad B_k:=B(\gamma_z(t_k),r_k).
		\]
		We also set $B_0:=B(z_0,\frac{\ell(\gamma_z)}{2c})$. As before, for $k\in\N$, we have 
		$\gamma_z(t_{k+1})\in\overline{B_k}$ and $B_{k+1}\subset\alpha B_k\subset\Om$, where $\alpha=2-\frac{1}{2c}$. 
		By the John condition, $d_\Om(z_0)\geq\frac{1}{c}\ell(\gamma_z)$, and so $B_0\subset B(z_0,\frac{d_\Om(z_0)}{2})=:B^*$ and 
		$\overline{B_0}\subset\Om$. 
		
Note that
\[
|Tu(z)-u_{B^*}|\le |u_{B_0}-u_{B^*}|+\sum_{k=0}^\infty|u_{B_k}-u_{B_{k+1}}|.
\]
By the doubling property of $\mu$, the inequalities $\ell(\gamma_z)\geq d(z,z_0)\geq d_\Om(z_0)$, 
and the $q$-Poincar\'e inequality, we have that
\[
|u_{B_0}-u_{B^*}|\lesssim \fint_{B^*}\!|u-u_{B^*}|\,d\mu
\lesssim 
d_\Om(z_0)\left(\fint_{B^*}\!g_u^{{q}}\,d\mu\right)^{\frac{1}{q}}.
\]
Moreover, fixing $0<\beta_0<1$, 
we see by an application of the $q$-Poincar\'e inequality followed by H\"older's inequality that
\begin{align*}
\sum_{k=0}^\infty|u_{B_k}-u_{B_{k+1}}|\lesssim \sum_{k=0}^\infty\fint_{\alpha B_k}|u-u_{\alpha B_k}|\,d\mu
  &\lesssim \sum_{k=0}^\infty r_k\, \left(\fint_{\alpha B_k}g_u^{q}\, d\mu\right)^{1/q}\\
  &\leq \left(\sum_{k=0}^\infty r_k^{\beta_0q'}\right)^{1/q'}\, 
         \left(\sum_{k=0}^\infty r_k^{q(1-\beta_0)}\fint_{\alpha B_k}g_u^{q}\, d\mu\right)^{1/q},
\end{align*}
where $q'=\tfrac{q}{q-1}$ is the H\"older conjugate of $q$. As $\gamma_z$ is a $c$-John curve with end points $z_0$, the John
center, and $z$, it follows that 
\[
r_k=\frac{1}{2c}\, \left(1-\frac{1}{2c}\right)^k\, \ell(\gamma_z)\le \frac{1}{2}\, \left(1-\frac{1}{2c}\right)^k\, d_\Om(z_0).
\]
Hence
\[
\sum_{k=0}^\infty r_k^{\beta_0q'}\lesssim d_\Om(z_0)^{\beta_0q'}\, \sum_{k=0}^\infty \left(1-\frac{1}{2c}\right)^{\beta_0q'k}
\approx d_\Om(z_0)^{\beta_0q'},
\]
and so
\[
\sum_{k=0}^\infty|u_{B_k}-u_{B_{k+1}}|
  \lesssim d_\Om(z_0)^{\beta_0}\, \left(\sum_{k=0}^\infty r_k^{q(1-\beta_0)}\fint_{\alpha B_k}g_u^{q}\, d\mu\right)^{1/q}
  \lesssim d_\Om(z_0)^{\beta_0}\, \left(\sum_{k=0}^\infty \frac{r_k^{q(1-\beta_0)}}{\mu(\alpha B_k)}\,
       \int_{\alpha B_k}g_u^{q}\, d\mu\right)^{1/q}.
\]
As in the proof of~\eqref{eq:trace-energy}, by
the doubling property of $\mu\vert_\Om$ and by~\eqref{eq:tau}, 
and with $\tau=8c^2$, we see that $\mu(B(z,9\tau^2c\, r_k)\cap\Om)\lesssim \mu(\alpha B_k)$.
Therefore,
\begin{align}\label{eq:Tr-Est1}
|Tu(z)-u_{B^*}|^q&\lesssim d_\Om(z_0)^q\, \fint_{B^*}\!g_u^{{q}}\,d\mu\,+\, d_\Om(z_0)^{\beta_0q}\,
\sum_{k=0}^\infty \frac{r_k^{q(1-\beta_0)}}{\mu(B(z,9\tau^2c\, r_k)\cap\Om)}\, \int_{\alpha B_k}\!g_u^{q}\, d\mu\notag\\
 & \lesssim \frac{d_\Om(z_0)^{q-p}}{\nu(F_\nu)}\, \int_{B^*}\!g_u^{q}\,d\mu\, +\,
 d_\Om(z_0)^{\beta_0q}\,\sum_{k=0}^\infty \frac{r_k^{q(1-\beta_0)-p}}{\nu(B(z,9\tau^2c\, r_k))}\, \int_{\alpha B_k}g_u^{q}\, d\mu,
\end{align}
where we used the upper bound~\eqref{eq:upper-bound-local} in the last step. We also used the estimate
\begin{align}\label{eq:nuF-Bstar}
0<\nu(F_\nu)= \nu(B(z_0,3d_{\Om}(z_0)))=\nu(B(z,6d_{\Om}(z_0)))
&\lesssim \frac{\mu(B(z,6d_{\Om}(z_0))\cap\Om)}{d_{\Om}(z_0)^p}\notag\\
&\le \frac{\mu(B(z_0,9d_{\Om}(z_0))\cap\Om)}{d_{\Om}(z_0)^p}\ \lesssim \frac{\mu(B^*)}{d_{\Om}(z_0)^p}.
\end{align}
As before, we have that $(2c+\alpha)r_k\geq d(w,z)$ whenever $w\in\alpha B_k$.

Integrating~\eqref{eq:Tr-Est1} over $z\in F_\nu$,  we obtain 
\begin{align*}
\int_{F_\nu}|Tu(z)-u_{B^*}|^q\, d\nu\lesssim &\,d_\Om(z_0)^{q-p}\int_{B^*}g_u^q\, d\mu\\
 +\,d_\Om(z_0)^{\beta_0q}&\,  \int_{F_\nu}\left(\sum_{k=0}^\infty\int_{\Om_{z_0}(c)\cap B(z_0,3d_\Om(z_0))}
  \!\frac{r_k^{q(1-\beta_0)-p}}{\nu(B(z,9\tau^2c\, r_k))}\, \chi_{\alpha B_k}(w)\, g_u(w)^{q}\, d\mu(w)\right) d\nu(z),
\end{align*}

keeping in mind that the chain of balls $\alpha B_k$ depend on $z\in F_\nu$. For each $z\in F_\nu$, setting 
\[
C_z:=\bigcup_{k=0}^\infty\alpha B_k,
\]
we see that if a point $w\in\Om_{z_0}(c)\cap B(z_0,3d_\Om(z_0))$ belongs to the cone $C_z$, then necessarily there
is some $k$ for which $w\in\alpha B_k$ and so
 $d(w,z)\le(2c+\alpha)r_k<\tfrac32\tau r_k$, from where it follows that
$d_\Om(w)\le d(w,z)<\tfrac32\tau r_k$. Furthermore, 
$d_\Om(w)> \tfrac{t_k}{c}-\alpha r_k=\tfrac{1}{2c}r_k$; recall that $\alpha=2-\tfrac{1}{2c}$. That is,
\begin{equation}\label{eq:comp-rk-vs-dOm}
\frac{1}{2c}r_k<d_\Om(w)\le d(w,z)<\frac{3}{2}\tau r_k.
\end{equation}
Therefore,
$B(w,3\tau c\, d_\Om(w))\subset B(z,9\tau^2c\, r_k)$. Hence,
by the bounded overlap property of the balls $\alpha B_k$ and by~\eqref{eq:comp-rk-vs-dOm}, we have that 
\begin{align*}
\int_{F_\nu}|Tu(z)-u_{B^*}|^q\, d\nu\lesssim &\,d_\Om(z_0)^{q-p}\int_{B^*}g_u^q\, d\mu\, +\, \\
  +\,d_\Om(z_0)^{\beta_0q}\, & \int_{F_\nu}\left(
  \int_{\Om_{z_0}(c)\cap B(z_0,3d_\Om(z_0))}\frac{d_\Om(w)^{q(1-\beta_0)-p}}{\nu(B(w,3\tau c\, d_\Om(w)))}\, \chi_{C_z}(w)\, g_u(w)^{q}\, d\mu(w)\right)d\nu(z)\\
\lesssim &\,d_\Om(z_0)^{q-p}\int_{B^*}g_u^q\, d\mu\, +\,\\
 +\,d_\Om(z_0)^{\beta_0q}\, \int_{F_\nu}&\left(
  \int_{\Om_{z_0}(c)\cap B(z_0,3d_\Om(z_0))}\frac{d_\Om(w)^{q(1-\beta_0)-p}}{\nu(B(w,3\tau c\,d_\Om(w)))}
\, \chi_{B(w,3\tau c d_\Om(w))}(z)\, g_u(w)^{q}\, d\mu(w)\right) d\nu(z)\\
\lesssim &\,d_\Om(z_0)^{q-p}\int_{B^*}g_u^q\, d\mu\, +\,\\
 +\,d_\Om(z_0)^{\beta_0q}\, &
  \int_{\Om_{z_0}(c)\cap B(z_0,3d_\Om(z_0))}\!\!\!\!\!\!\!\!\!g_u(w)^q\, \left(\int_{F_\nu}\frac{d_\Om(w)^{q(1-\beta_0)-p}}{\nu(B(w,3\tau c\,d_\Om(w)))}\, \chi_{B(w,3\tau c d_\Om(w))}(z)\, d\nu(z)\right)d\mu(w)\\
= &\,d_\Om(z_0)^{q-p}\int_{B^*}g_u^q\, d\mu\, +\, \\
 +\,d_\Om(z_0)^{\beta_0q}\, &
  \int_{\Om_{z_0}(c)\cap B(z_0,3d_\Om(z_0))}\!\!\!\!\!\!\!\!\!g_u(w)^q\, d_\Om(w)^{q(1-\beta_0)-p}\,d\mu(w).
\end{align*}
As $q>p$, we can chose $\beta_0$ small enough so that 
$p<q(1-\beta_0)$.
As $d_\Om(w)\le 4d_\Om(z_0)$ when $w\in B(z_0,3d_\Om(z_0))$, 
and as $B^*\subset \Om_{z_0}(c)\cap B(z_0,3d_\Om(z_0))$, we see that
\begin{align*}
\int_{F_\nu}|Tu(z)-u_{B^*}|^q\, d\nu&\lesssim d_\Om(z_0)^{q-p}\int_{B^*}g_u^q\, d\mu\, +\,
 d_\Om(z_0)^{\beta_0q} \, d_\Om(z_0)^{q(1-\beta_0)-p}\, \int_{\Om_{z_0}(c)\cap B(z_0,3d_\Om(z_0))}\!\!\!\!\!\!\!\!\!g_u(w)^q\, d\mu(w)\\
 &\lesssim d_\Om(z_0)^{q-p}\, \int_{\Om_{z_0}(c)\cap B(z_0,3d_\Om(z_0))}\!\!\!\!\!\!g_u^q\, d\mu,
\end{align*}
and so by~\eqref{eq:nuF-Bstar},
\begin{align*}
\int_{F_\nu}\!|Tu(z)|^q\,d\nu&\lesssim \nu(F_\nu)\,|u_{B^*}|^q + d_{\Om}(z_0)^{q-p}\int_{\Om_{z_0}(c)}\! g_u^{q}\,d\mu\\
&\leq \frac{\nu(F_\nu)}{\mu(B^*)}\int_{\Om_{z_0}(c)}|u|^q\,d\mu + d_{\Om}(z_0)^{q-p}\int_{\Om_{z_0}(c)}\! g_u^{q}\,d\mu\\
&\lesssim \frac{1}{d_{\Om}(z_0)^p}\int_{\Om_{z_0}(c)}|u|^q\,d\mu+ d_{\Om}(z_0)^{q-p}\int_{\Om_{z_0}(c)}\! g_u^{q}\,d\mu.
\end{align*}
Therefore, we reach the desired conclusion. 
\end{proof}

Now we complete this section by providing the proof of Theorem~\ref{thm:main2}.

\begin{proof}[Proof of Theorem~\ref{thm:main2}]
With $\nu_F$ the Frostman measure constructed in Section~\ref{Sec:3}, we know from~\eqref{eq:Frostman-Control} that 
$\nu_F$ satisfies the dominancy condition~\eqref{eq:upper-bound-local}. Moreover, from the proof of Proposition~\ref{prop:Frostman},
in particular from~\eqref{eq:nontriv-Frost}, we know that for each $w\in P_\infty$ and $r>0$ that $\nu_F(B(w,r))>0$.
The proof is now completed by a direct application of Theorem~\ref{thm:subsidiary1} above.
\end{proof}

\section{Concluding remarks}\label{Sec:6}

In summary, from Theorem~\ref{thm:main} we know that any domain that satisfies~\eqref{eq:Ass1} has a sufficiently large
visible boundary corresponding to the $c$-John condition for suitable $c>1$. From Theorem~\ref{thm:main2} we also
know in addition that if the restriction of the measure and the metric on the domain both satisfy doubling property
and support a $p$-Poincar\'e inequality, then there is a bounded linear trace operator from the Newton-Sobolev
class of the domain to the corresponding Besov class on the visible boundary, provided an upper 
control~\eqref{eq:upper-bound-local} of the measure on the visible boundary also holds. To obtain a possible 
trace on the entire boundary of the domain, one might wish to move the point $z_0$ around the domain, and 
also increase the John constant $c$.
	
It is then expedient to know whether 
$\partial\Om=\bigcup_{c>1}\bigcup_{z_0\in\Om}\Om_{z_0}(c)$. As external cusp domains show, the above
is not always possible. So the next natural question to ask is whether, with the choice of $t$ that 
makes~\eqref{eq:Ass1} valid, we have
$\mathcal{H}^{-p}(\partial\Om\setminus \bigcup_{c>1}\bigcup_{z_0\in\Om}\Om_{z_0}(c))=0$. The next example shows
that this is not always the case, either.
	
\begin{example}
Let $X$ be the Euclidean plane, equipped with the Euclidean metric and the $2$-dimensional Lebesgue measure.
We construct a bounded domain $\Om\subset X$ as follows. For each positive integer $n$ we set
\[
\alpha_n:=\begin{cases} \frac{1}{n}-\pi &\text{ if }2\mid n,\\
\frac{1}{n} &\text{ if }2\nmid n,\end{cases}
\]
and
\[
\beta_n:=\begin{cases} \pi-\frac{1}{n} &\text{ if }2\mid n,\\
2\pi-\frac{1}{n} &\text{ if }2\nmid n.\end{cases}
\]
We now set $E_n:=\{(1-\tfrac1n)e^{i\theta}\, :\, \alpha_n\le \theta\le \beta_n\}$ for each positive integer $n$, and 
$\Om=\mathbb{D}\setminus\bigcup_{n\in\N}E_n$, where $\mathbb{D}$ is the unit disk in $\R^2$, centered
at $0$. Then note that
\[
\partial\Om\setminus \bigcup_{c>1}\bigcup_{z_0\in\Om}\Om_{z_0}(c)=\partial\mathbb{D}.
\]
Moreover, with $t=1$ the condition~\eqref{eq:Ass1} does hold, and as $\R^2$ satisfies a $1$-Poincar\'e inequality, we 
can consider any choice of $p$ with $1<p<2$.
\end{example}

Note also that we require the domain itself to support a Poincar\'e inequality in Theorem~\ref{thm:main2}. As the
slit disk $\Omega=\{(x,y)\in\R^2\, :\, x^2+y^2<1\}\setminus([0,1]\times\{0\})$ shows, even for John domains
trace results do not make sense without additional geometric constraints on the domain. The verification of 
a Poincar\'e inequality seems to be a reasonable such constraint. 
Given this, one might wonder whether any domain that satisfies the requirements of Theorem~\ref{thm:main2}
must necessarily be a John domain itself. This is not true, as the following example shows.

\begin{example}
Note that the unit disk $\mathbb{D}=\{(x,y)\in\R^2\, :\, x^2+y^2<1\}$, equipped with the $2$-dimensional Lebesgue
measure and the Euclidean metric, is doubling and supports a $p$-Poincar\'e inequality for each $1\le p<\infty$.
Although $\mathbb{D}$ is itself a John domain (and indeed, is a uniform domain), it is not impossible to
remove a closed subset $K\subset\mathbb{D}$ from $\mathbb{D}$ in order to destroy the $c$-John condition
for each $c>1$, whereas the resulting
domain $\Om=\mathbb{D}\setminus K$ is still doubling and supports the same Poincar\'e inequality as long as $K$ has zero $p$-capacity. Thus the John
condition is not a consequence of the support of Poincar\'e inequality. 
\end{example}

The above example reveals a further line of enquiry: is there a
way to modify the notion of visible boundary so that it does not rely solely on the John condition and leads to a
trace theorem that extends Theorem~\ref{thm:main2} so that obstacles such as capacity zero sets that destroy the John condition
are overlooked?

	\noindent {\bf Address:} \\
	
	\noindent S.E.-B.: University of Jyv\"askyl\"a, Department of Mathematics and Statistics, P.O. Box 35 (MaD), FI-40014 University of Jyv\"askyl\"a, Finland \\
	\noindent E-mail: S.E-B.: {\tt sylvester.d.eriksson-bique@jyu.fi}\\
	
	\vskip .2cm
	
	\noindent R.G.: Department of Mathematical Sciences, P.O.~Box 210025, University of Cincinnati, Cincinnati, OH~45221-0025, U.S.A.\\
	\noindent Department of Mathematics, Physics and Geology, Cape Breton University, Sydney, NS~B1Y3V3, Canada.\\
	\noindent E-mail: R.G.: {ryan\textunderscore gibara@cbu.ca}\\
	
	\vskip .2cm
	
	\noindent R.K.: Department of Mathematics and Systems Analysis, Aalto University, P.O. Box 11100, FI-00076 Aalto, Finland.
	\\
	\noindent E-mail:  R.K.: {\tt riikka.korte@aalto.fi}\\
	
	\vskip .2cm
	
	\noindent N.S.: Department of Mathematical Sciences, P.O.~Box 210025, University of Cincinnati, Cincinnati, OH~45221-0025, U.S.A.\\
	\noindent E-mail:  N.S.: {\tt shanmun@uc.edu}\\


\begin{thebibliography}{AA}
		\frenchspacing
		\bibitem{Ash} H. Aikawa, N. Shanmugalingam: 
		\emph{Carleson-type estimates for $p$-harmonic functions
		and the conformal Martin boundary of John domains in metric measure spaces.}
		 Michigan
		Math. J. {\bf 53} (2005), no. 1, 165-188.
		\bibitem{Az} J. Azzam:
		\emph{Accessible parts of the boundary for domains with lower content regular complements.}
		Ann. Acad. Sci. Fenn. Math. {\bf 44} (2019), no. 2, 889--901.
		\bibitem{Besov} O. V. Besov:
		\emph{Investigation of a class of function spaces in connection with imbedding and extension theorems.}
		Trudy Mat. Inst. Steklov. {\bf 60} 1961 42--81.
		\bibitem{BBbook} A. Bj\"orn, J. Bj\"orn:
		\emph{Nonlinear potential theory on metric spaces.}
		EMS Tracts in Mathematics, {\bf 17}, European Mathematical Society (EMS), Z\"urich, 2011. xii+403 pp.
		\bibitem{BBS} A. Bj\"orn, J. Bj\"orn, N. Shanmugalingam:
		\emph{Extension and trace results for doubling metric measure spaces and their hyperbolic fillings.}
		J. Math. Pures Appl. (9) {\bf 159} (2022), 196--249.
		\bibitem{CKKSS} L. Capogna, J. Kline, R. Korte, N. Shanmugalingam, M. Snipes:
		\emph{Neumann problems for $p$-harmonic functions, and induced nonlocal operators in metric measure spaces.}
		Amer. J. Math., {\bf 147} (2025), no. 6, 1653--1711.
		\bibitem{E-B} S. Eriksson-Bique:
		\emph{Alternative proof of Keith-Zhong self-improvement and connectivity.}
		Ann. Acad. Sci. Fenn. Math. {\bf 44} (2019), no. 1, 407--425.
		\bibitem{EGKS} S. Eriksson-Bique, R. Gibara, R. Korte, N. Shanmugalingam:
		\emph{Sharp Hausdorff content estimates for accessible boundaries of domains in metric measure spaces of controlled geometry.}
                 Trans. Amer. Math. Soc. {\bf 11} (2024), 1435--1461.
		\bibitem{Gag} E. Gagliardo:
		\emph{Caratterizzazioni delle tracce sulla frontiera relative ad alcune classi di funzioni in $n$ variabili.}
		Rend. Sem. Mat. Univ. Padova {\bf 27} (1957), 284--305.
		\bibitem{GK} R. Gibara, R. Korte:
		\emph{Accessible parts of the boundary for domains in metric measure spaces.}
		Ann. Fenn. Math. {\bf 47} (2022), 695--706.
		\bibitem{GKS} R. Gibara, R. Korte, N. Shanmugalingam:
		\emph{Solving a Dirichlet problem for unbounded domains via a conformal transformation.}
		Math. Ann. {\bf 389} (2024), no. 3, 2857--2901.
		\bibitem{GS2} R. Gibara, N. Shanmugalingam:
		\emph{Trace and extension theorems for homogeneous Sobolev and Besov spaces for unbounded uniform domains in metric measure spaces.}
		Proc. Steklov Inst. Math. {\bf 323} (2023), no. 1, 101--119.
		\bibitem{GoKS} A. Gogatishvili, P. Koskela, N. Shanmugalingam:
		\emph{Interpolation properties of Besov spaces defined on metric spaces.}
		Math. Nachr. {\bf 283} (2010), no. 2, 215--231.
		\bibitem{Gris} P. Grisvard:
		\emph{Spazi di tracce e applicazioni.}
		Rend. Mat. (6) {\bf 5} (1972), 657--729.
		\bibitem{HKM} J. Heinonen, T. Kilpel\"ainen, O. Martio:
		\emph{Nonlinear potential theory of degenerate elliptic equations.}
		Dover Publications, reprint edition 2006.
		\bibitem{HK} J. Heinonen, P. Koskela:
		\emph{Quasiconformal maps in metric spaces with controlled geometry.}
		Acta Math. {\bf 18} (1998), 1--61.
		\bibitem{HKSTbook} J. Heinonen, P. Koskela, N. Shanmugalingam, J. T. Tyson:
		\emph{Sobolev spaces on metric measure spaces: an approach based on upper gradients.}
		New Mathematical Monographs {\bf 27}, 
		Cambridge University Press (2015).
		\bibitem{JW} A. Jonsson, H. Wallin:
		\emph{Function spaces on subsets of $\R^n$.}
		Math. Rep. {\bf 2} (1984), no. 1, xiv+221 pp.
		\bibitem{KZ} S. Keith, X. Zhong:
		\emph{The Poincar\'e inequality is an open ended condition.}
		Ann. of Math. (2) {\bf 167} (2008), 575--599.
		\bibitem{KL} P. Koskela, J. Lehrbäck:
		\emph{Weighted pointwise Hardy inequalities.}
		J. Lond. Math. Soc. (3) {\bf 79} (2009), 757--779.
		\bibitem{KNN} P. Koskela, D. Nandi, A. Nicolau:
		\emph{Accessible parts of boundary for simply connected domains.}
		Proc. Amer. Math. Soc. {\bf 146} (2018), no. 8, 3403--3412.
		\bibitem{L} J. Lehrbäck: 
		\emph{Weighted Hardy inequalities beyond Lipschitz domains.} 
		Proc. Amer. Math. Soc. (5) {\bf 142} (2014), 1705--1715.
		\bibitem{Mal} L. Mal\'{y}:
		\emph{Trace and extension theorems for Sobolev-type functions in metric spaces.}
		Preprint (2017), {\tt https://arxiv.org/abs/1704.06344}
		\bibitem{RY} T. Runst, A. Youssfi:
		\emph{Boundary value problems for Waldenfels operators.}
		Indiana Univ. Math. J. {\bf 54} (2005), no. 1, 237--255.
		\bibitem{Strichartz} R. S. Strichartz:
		\emph{Boundary values of solutions of elliptic equations satisfying $H^p$ conditions.}
		Trans. Amer. Math. Soc. {\bf 176} (1973), 445--462.
		\bibitem{Tyu1} A. I. Tyulenev:
		\emph{Traces of Sobolev spaces to irregular subsets of metric measure spaces.}
                  Mat. Sb. {\bf 214} (2023), no. 9, 58--143; translation in
                  Sb. Math. {\bf 214} (2023), no. 9, 1241--1320.
                  \bibitem{Tyu2} A. I. Tyulenev:
                  \emph{Traces of Sobolev spaces on piecewise Ahlfors-David regular sets.}
                  Mat. Zametki {\bf 114} (2023), no. 3, 404--434; translation in
                  Math. Notes {\bf 114} (2023), no. 3-4, 351--376.
		\bibitem{VK} A. Volberg, S. V. Konyagin:
		\emph{On measures with the doubling condition.}
		Izv. Akad. Nauk SSSR Ser. Mat. {\bf 51} (1987), no. 3, 666--675; translation in Math. USSR-Izv. {\bf 30} (1988), no. 3, 629--638.
	\end{thebibliography}
\end{document}